\documentclass[pdflatex,sn-mathphys-num]{sn-jnl}

\usepackage{graphicx}%
\usepackage{multirow}%
\usepackage{amsmath,amssymb,amsfonts}%
\usepackage{amsthm}%
\usepackage{mathrsfs}%
\usepackage[title]{appendix}%
\usepackage{xcolor}%
\usepackage{textcomp}%
\usepackage{manyfoot}%
\usepackage{booktabs}%
\usepackage{algorithm}%
\usepackage{algorithmicx}%
\usepackage{algpseudocode}%
\usepackage{listings}%

\theoremstyle{thmstyleone}%
\newtheorem{theorem}{Theorem}
\newtheorem{proposition}[theorem]{Proposition}%

\theoremstyle{thmstyletwo}%

\theoremstyle{thmstylethree}%

\begin{document}

\title[How delay, isolation and vaccination shape epidemic waves: A bifurcation approach]{How delay, isolation and vaccination shape epidemic waves: A bifurcation approach {\color{black}in mathematical epidemiology}}


\author[1]{\fnm{Mehdi } \sur{Bouziane}}

\author[2]{\fnm{Silia} \sur{Bercisse}}

\author*[2]{\fnm{Abdennasser} \sur{Chekroun}}

\author[3]{\fnm{Simon} \sur{Girel}}

\affil[1]{\orgdiv{Department of Mathematics}, \orgname{Higher Normal School of Mostaganem}, \city{Mostaganem}, \postcode{27000}, \country{Algeria}}

\affil[2]{\orgdiv{Laboratoire d'Analyse Nonlin\'{e}aire et Math\'{e}matiques Appliqu\'{e}es}, \orgname{University of Tlemcen}, \city{Tlemcen}, \postcode{13000}, \country{Algeria}}

\affil[3]{\orgdiv{Laboratoire J.A.Dieudonn\'{e}}, \orgname{Universit\'{e} C\^{o}te d'Azur}, \orgaddress{\street{UMR CNRS-UNS 7351}, \city{Nice}, \postcode{06000}, \country{France}}}


\abstract{This research paper introduces an SQIR-V epidemic model to investigate the transmission of infectious diseases. Particular attention is paid to the roles of vaccination and quarantine (incorporating physical distancing interventions) in protecting susceptible individuals. {\color{black} The model features nonlinear transition rates that depend on the history of infection, allowing the emergence of periodic solutions.} We calculate the basic reproduction number, $ R_{0}$, and analyze the local asymptotic stability of the equilibrium points. Additionally, we demonstrate that the disease-free equilibrium is globally asymptotically stable when $R_{0}\leq 1$. The study further explores the existence of periodic solutions through a Hopf bifurcation, showing the occurrence of epidemic waves. A condition was derived to determine the direction of the crossing of the imaginary axis. We finish by presenting some numerical simulations to illustrate how vaccination and isolation delays influence disease dynamics. Those findings highlight potential areas for further research and validation.
}

\keywords{Epidemic model, isolation and vaccination, {\color{black}Nonlinear incidence with memory}, basic reproduction number, delay differential system, global stability, Hopf bifurcation analysis}


\pacs[MSC Classification]{34K13, 34K20, 37N25, 92D30}

\maketitle

\section{Introduction}

Infectious diseases have profoundly affected the human life, causing the death or suffering of millions of the people every year due to various infections \cite{r010}. Managing those diseases has
become increasingly complex over time, as the infectious diseases evolve and spread globally once they become widespread \cite{r29}. During the past 200 years, 335 cases of the emerging diseases have been recorded \cite{r27}, including seven cholera epidemics, four new strains of the influenza virus, as well as the tuberculosis, HIV, SARS, Covid-19, Lyme disease, and Ebola, causing at least 100 million deaths \cite{r28cob}. Infectious diseases have remained one of the leading causes of death worldwide \cite{r010}. They have had a significant impact on both the global economy and public health, emphasizing the importance of effective prevention and control strategies to reduce their incidence \cite{r011}.

 Vaccination has been one of the most effective medical strategies for preventing and controlling the infectious diseases \cite{r05}. It has been estimated that the vaccines save approximately six million lives each year from a variety of diseases  \cite{r33}. Researchers have been continuously working on the development of vaccines and antibiotics to protect the populations around the world from the infectious diseases \cite{r34, r35}. For example, there are various vaccines available to prevent different diseases: these include the seasonal influenza vaccine for the flu, Pfizer-BioNTech and Moderna vaccines for Covid-19, BCG for tuberculosis, Dengvaxia for dengue fever, MMR for measles, and Ervebo for Ebola \cite{r01,r02,r03,r04}. However, some vaccines do not provide complete protection \cite{r32,r31}. In addition, the prompt implementation of quarantine and isolation measures is essential for the effective management of epidemic outbreaks, as any delay can significantly worsen the situation \cite{r44}. On the other hand, based on the pioneering model developed by Kermack and McKendrick \cite{r43}, mathematical modeling has become a crucial tool to understand the dynamics of infectious diseases and to develop effective epidemic control strategies \cite{AlexSIam,r36,Hathou2022, r37,Yang2016Taiw}. Numerous models have been proposed in the literature. For instance, in \cite{r38,r24}, the authors developed models for Covid-19 incorporating both vaccination and quarantine strategies. The work in \cite{r41} studied the co-infection dynamics of Covid-19 and influenza by considering the treatment and hospitalization compartments, and \cite{r40} modeled tuberculosis transmission using an SEIR framework with three control variables, namely education, vaccination and treatment. The work \cite{r42} investigated the impact of vaccination and quarantine on the spread of Ebola.

Many infectious diseases exhibit delays in their transmission dynamics, such as latency periods and temporary immunity phases \cite{r45, r46, salh2024, r47, r48, r49,Yang2018MMAS}. These delays are inherent in natural and biological systems and significantly affect the precision of epidemic predictions, as well as the effectiveness of control strategies \cite{r50}. To better capture real-world scenarios, researchers have increasingly incorporated time delays into mathematical models of infectious diseases. For example, in \cite{r54}, the authors investigated how media coverage and time delays influence the control of infectious diseases using an SIS model. The authors in \cite{AdchkDCDB,r53} studied epidemic dynamics involving multiple populations using discrete and distribute delay differential-difference equations. The results in \cite{AdchkDCDB} provide quantitative information on the relative effectiveness of protecting susceptible individuals versus isolating infectious individuals, highlighting the critical role of the duration of the intervention in epidemic control. The author in \cite{r52} examined the impact of time delays and waning vaccine-induced immunity on disease transmission dynamics. In addition, the work \cite{r55} proposed an epidemic model that incorporates pulse vaccination and distributed time delays. Similarly, the authors in \cite{r56} developed a delayed SEIQR model featuring both pulse vaccination and quarantine. Finally, in \cite{r57}, the paper introduced a novel model that integrates quarantine, isolation, and an imperfect vaccine to study disease transmission dynamics.

This research aims to develop a mathematical model to examine the transmission dynamics of infectious diseases. The model incorporates various factors, including vaccination and quarantine measures for susceptible individuals. It is formulated as a system of distributed delay differential equations that account for the infection history. We extend the work of \cite{kunireccuwave} by integrating additional a vaccinated compartment, regulated by distributed delay dynamics, that enhance both its mathematical perspectives and practical relevance. We perform a detailed analysis of the dynamic properties of the model, with particular emphasis on the basic reproduction number $R_0$, and determine the conditions under which a Hopf bifurcation may occur. The resulting periodic solutions represent recurrent epidemic waves, which contributes to our understanding of disease spread and control. By integrating distributed delays, quarantine strategies, and vaccination effects, the model offers a more realistic representation of infectious disease transmission and management in real-world scenarios.

The paper is structured as follows. In Section \ref{mathmode}, a model that incorporates quarantine measures and the effect of the vaccine is introduced. Section \ref{basexiseui} presents the basic
reproduction number $R_{0}$, which serves as a threshold parameter for disease spread. Section \ref{ststeadyDFE} focuses on establishing the local and global stability of the disease-free equilibrium. Section \ref{staendmiequ} is devoted to the stability around the endemic equilibrium. In Section \ref{hopfpart}, the analysis for a Hopf bifurcation at endemic equilibrium (when $R_{0}>1$) is discussed, giving rise to oscillations. Section \ref{simu} provides numerical examples that demonstrate how various parameters can lead to periodic epidemic waves within the Hopf bifurcation. {\color{black} Finally, we conclude with a brief discussion in Section \ref{disc}.}

\section{Mathematical model}\label{mathmode}

\textcolor{black}{Early epidemic models primarily used classical SIR frameworks without accounting for environmental or public-health factors. Later models included vaccination, quarantine, and time delays, but often focused on single strategies or assumed perfect vaccination, neglecting their combined effects and delays. In this paper, we propose a novel model integrating quarantine, imperfect vaccination, and distributed delays, offering new insights worthy of exploration.}
 We develop a distributed-delay SQIR-V model to analyze epidemic spread. The total population $N$ is subdivided into five distinct groups: susceptible $S$, vaccinated $V$, quarantine $Q$ (individuals following precautionary health measures), infectious $I$ and recovered $R$. Then, at any time $t\geq0$, $N$ is defined by: 
\begin{equation*}
	N(t)=S(t)+V(t) +Q(t)+I(t)+R(t).
\end{equation*}
Next, we use the following notation
\[
(f*I)(t) :=\int_{0}^{+\infty}f\left( s \right)
I\left( t-s \right) ds,
\]
which represents the amount of reliable information available about the infection history of individuals at the time $t$. The function $f$ acts as a general kernel characterizing the influence of past infections. In Figure \ref{diadig1}, we show the diagram of our model that illustrates the transmission dynamics of the epidemic, including the transition and the interactions between individuals or population groups over time.

\begin{figure}[H]
	\begin{center}
		\includegraphics[width=11cm]{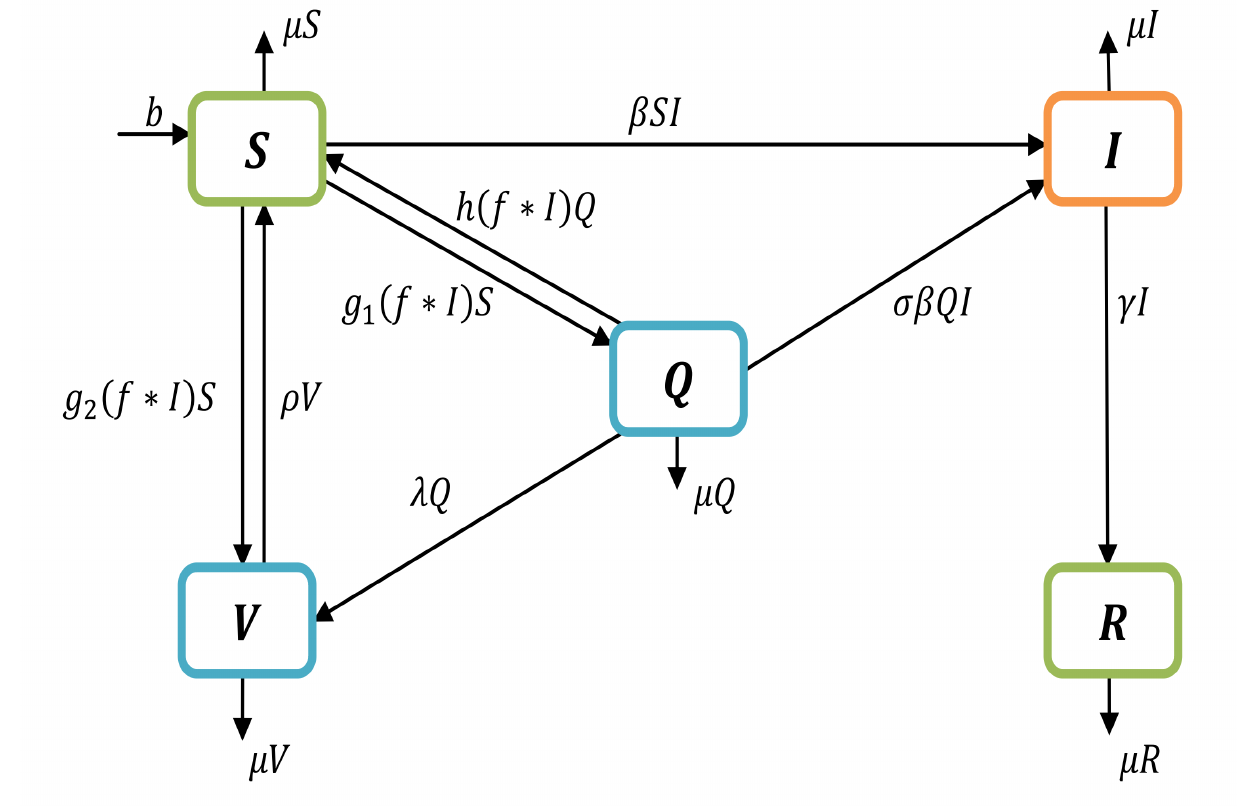} 
		\caption{ Schematic diagram illustrating the transmission dynamics between population groups.} 	\label{diadig1}  
	\end{center}
\end{figure}

The meaning of each parameter is as follows: $b$ is the birth rate,
$\beta$ is the infection rate,
$\gamma$ is the recovery rate,
$ \lambda$ is the vaccination rate in compartment $Q$,
$\sigma $ is the degree of defect of quarantine (imperfection) with $0\leq\sigma<1 $,
$\mu $ is the mortality rate.
$g_1(f*I)$ is the transition rate from compartment $S$ to compartment $Q$,
$g_2(f*I)$ is the transition rate from compartment $S$ to compartment $V$,
$h(f*I)$ is the transition rate from compartment $Q$ to compartment $S$,
$\rho$ is the loss rate of of vaccination-induced immunity.

Infection can be transmitted through contact with infected people. Susceptible individuals can be quarantined. However, quarantined individuals may eventually become infected. Vaccination is also incorporated into the model, targeting both susceptible and quarantined populations at distinct rates. It is also noted that individuals can recover or die and that natural mortality affects all groups involved. The key difference between vaccination and isolation in the model lies in their effectiveness: vaccination is assumed to offer full protection while isolation only lowers the infection rate.

The model is governed by the following system of distributed delay differential equations, for $t>0$,
\begin{equation*}
	\left\{
	\begin{array}{lll}
		S^{\prime }\left( t\right) &=& b-\beta S\left( t\right) I\left( t\right)
		-g_{1}\left(f* I\right)S(t)-g_{2}\left(f* I\right)S(t)-\mu S\left( t\right) \vspace{0.2cm}\\
		&& +h\left(f* I\right) Q\left( t\right)
		+\rho V\left( t\right) ,  \vspace{0.2cm}\\ 
		V^{\prime }\left( t\right) &=& g_{2}\left(f* I\right)S(t) +\lambda Q\left( t\right) -\left( \rho+\mu
		\right) V\left( t\right) ,\vspace{0.2cm} 
		\\ 
		Q^{\prime }\left( t\right) &=&g_{1}\left(f* I\right)S(t) -h\left(f* I\right) Q\left( t\right) -\sigma \beta Q\left( t\right) I\left( t\right)
		-(\lambda  +\mu) Q\left(
		t\right), \vspace{0.2cm}
		\\ 
		I^{\prime }\left( t\right) &=&\beta S\left( t\right) I\left( t\right) +\sigma
		\beta Q\left( t\right) I\left( t\right) -\left( \gamma +\mu \right) I\left(
		t\right),
	\end{array}
	\right.\label{}
\end{equation*}
and 
\[
R^{\prime }\left( t\right) =\gamma I\left( t\right) -\mu R\left( t\right),
\]
with nonnegative initial conditions. The model comes with the following assumptions:\\
(A1) : $b$, $\beta$, $\mu$, $\lambda$, $\rho$ and $\gamma$ are positive constants.\\
(A2) : Quarantine might not be completely effective, then $0\leq \sigma <1$.
\\
(A3) : We assume that $g_1$, $g_2$ are increasing functions and $h$ is a decreasing function with $\lim_{x\to  +\infty}  h(x)=0$.  More specifically, we consider, for $q_1,q_2,\delta,\alpha>0$,
$$g_i(f*I)(t)=q_i\int_{0}^{+\infty}f(s)  I(t-s)ds \quad \text{and} \quad h(f*I)(t)=\frac{\delta}{1+\alpha\int_{0}^{+\infty}f(s)  I(t-s) ds}.$$
where $q_{1,2}$ and $\alpha$ stand for the sensitivity to the history of the infected individuals and $\delta$ is the maximum theoretical value of the rate $h(f*I)$.
\\
(A4) : The kernel function $f$ is nonnegative with compact support on $\mathbb{R}^+$, satisfying
$$\int_{0}^{+\infty}f(s) ds=1,$$
and that there exists $\tau>0$ such that
$$\left(f* I\right)(t)=\int_{0}^{+\infty}f(s) I(t-s) ds=\int_{0}^{\tau}f(s) I(t-s) ds.$$

Since the equations for $S$, $Q$, $V$ and $I$ in our system are independent of $R$, the equation for $R$ can be omitted and we focus on the following reduced system, for $t>0$,
\begin{equation}
	\left\{
	\begin{array}{lll}
		S^{\prime }\left( t\right) &=& b-\beta S\left( t\right) I\left( t\right)
		-g_{1}\left(f* I\right)S(t)-g_{2}\left(f* I\right)S(t)-\mu S\left( t\right) \vspace{0.2cm}\\
		&& +h\left(f* I\right) Q\left( t\right)
		+\rho V\left( t\right) ,  \vspace{0.2cm}\\ 
		V^{\prime }\left( t\right) &=& g_{2}\left(f* I\right)S(t) +\lambda Q\left( t\right) -\left( \rho+\mu
		\right) V\left( t\right) ,\vspace{0.2cm} 
		\\ 
		Q^{\prime }\left( t\right) &=&g_{1}\left(f* I\right)S(t) -h\left(f* I\right) Q\left( t\right) -\sigma \beta Q\left( t\right) I\left( t\right)
		-(\lambda  +\mu) Q\left(
		t\right), \vspace{0.2cm}
		\\ 
		I^{\prime }\left( t\right) &=&\beta S\left( t\right) I\left( t\right) +\sigma
		\beta Q\left( t\right) I\left( t\right) -\left( \gamma +\mu \right) I\left(
		t\right),
	\end{array}
	\right.\label{sysmprinciple}
\end{equation}
with a nonnegative initial density $\phi:=(\phi_1,\phi_2,\phi_3,\phi_4)^T$ given by, for all $\theta\in[-\tau,0]$,
\begin{equation*}
	S(\theta)=\phi_1(\theta),\ V(\theta)=\phi_2(\theta), \ Q(\theta)=\phi_3(\theta), \ I(\theta)=\phi_4(\theta).
\end{equation*}

We will give some details regarding the validity of the system. Let $E$ be the set of continuous functions from $[-\tau,0]$ to $\mathbb{R}^4 $, equipped with the norm:
\[\|\varphi\|_{E}:= \underset{\theta\in[-\tau,0]}{\sup} \|\varphi(\theta)\|_\infty, \qquad \varphi \in E,
\]
where $\|\cdot\|_\infty$ denotes the infinite norm in $\mathbb{R}^4 $ defined as
\begin{equation*}
	\|v\|_\infty= \max(|v_1|,|v_2|,|v_3|,|v_4|), \qquad \forall v=(v_1,v_2,v_3,v_4)^T \in \mathbb{R}^4.
\end{equation*}

For $u(t):=(S(t),V(t),Q(t),I(t))$, we use the classical notation $u_t(\theta):=u(t+\theta)$ to rewrite our system as $u'(t)=F(u_t)$, where for $\varphi= (
\varphi_1,\varphi_2,\varphi_3,\varphi_4)^T\in E$,
{\footnotesize
	\begin{align*}
		F(\varphi) =
		\fontsize{9}{10}\selectfont{ \begin{pmatrix}
				b - \beta \varphi_1(0)\varphi_4(0) - g_1(f * \varphi_4) \varphi_1(0) - g_2(f * \varphi_4) \varphi_1(0) - \mu \varphi_1(0) \\
				+ h(f * \varphi_4) \varphi_3(0) + \rho \varphi_2(0) \vspace{0.3cm}\\
				g_2(f * \varphi_4) \varphi_1(0) + \lambda \varphi_3(0) - (\mu + \rho) \varphi_2(0) \vspace{0.3cm} \\
				g_1(f * \varphi_4) \varphi_1(0) - h(f * \varphi_4) \varphi_3(0) - \sigma \beta \varphi_3(0) \varphi_4(0) - (\lambda + \mu) \varphi_3(0) \vspace{0.3cm}\\
				\beta \varphi_1(0) \varphi_4(0) + \sigma \beta \varphi_3(0) \varphi_4(0) - (\gamma + \mu) \varphi_4(0)
		\end{pmatrix}}.
	\end{align*}
}

We define the solution space for system (\ref{sysmprinciple}) as:
\begin{equation*}
	\Gamma:=\biggl\{\varphi=(\varphi_1,\varphi_2,\varphi_3,\varphi_4)^T\in E_{+}, \quad  0\leq\sum\limits_{i=1}^{4}\varphi_i(\theta)\leq b/\mu, \quad-\tau\leq \theta \leq 0\biggr\},
\end{equation*}
with $E_{+}$ is the positive cone of $E$. Note that $b/\mu$ is threshold beyond which the population $N$ decreases.

To prove the existence and uniqueness of solutions, we demonstrate that $F$ is Lipschitz continuous on $\Gamma$, see \cite{Halebook}. Let $u,v \in\Gamma$, we have
\begin{align*}
	&\big| F_1(u)-F_1(v)\big|\vspace{0.2cm}\\
	&\leq \beta \big| v_1(0)v_4(0)-u_1(0)u_4(0)\big|+\mu \big| v_1(0)-u_1(0)\big|+\rho \big| u_2(0)-v_2(0)\big|\\&+ \big| g_1(f*v_4)v_1(0)-g_1(f*u_4)u_1(0)\big|+ 
	\big| g_2(f*v_4)v_1(0)-g_2(f*u_4)u_1(0)\big| \\&+\big| h(f*u_4)u_3(0)-h(f*v_4)v_3(0)\big|,\vspace{0.2cm}
	\\&
	\leq\beta u_4(0) \big| u_1(0)-v_1(0)\big|+\beta v_1(0) \big| u_4(0)-v_4(0)\big|+\mu \big| u_1(0)-v_1(0)\big|\\&+\rho \big| u_2(0)-v_2(0)\big|
	+ g_1(f*u_4)\big| u_1(0)-v_1(0)\big|+v_1(0)\big| g_1(f*u_4)-g_1(f*v_4)\big|
	\\&+g_2(f*u_4)\big| u_1(0)-v_1(0)\big|
	+v_1(0)\big| g_2(f*u_4)-g_2(f*v_4)\big|\\&+h(f*u_4)\big| u_3(0)-v_3(0)\big|+v_3(0)\big| h(f*u_4)-h(f*v_4)\big|.
\end{align*}
We put $N=b/\mu$ and using $\int_{0}^{+\infty}f(s)\,ds= 1$, we obtain
\begin{align*}
	\big| F_1(u)-F_1(v)\big| &\leq \big(\beta N + \mu + q_1 N + q_2 N\big)\big| u_1(0)-v_1(0)\big|\\ &  \ \ \ +\beta N \big| u_4(0)-v_4(0)\big|+\rho \big| u_2(0)-v_2(0)\big| \\
	& \ \ \ +\delta\big| u_3(0)-v_3(0)\big|
	+  N (q_1+q_2) \int_{0}^{+\infty}f(s) \big| u_4(-s)-v_4(-s)\big| ds\\
	& \ \ \ + N \Big|\frac{\delta}{1+\alpha\int_{0}^{+\infty}f(s)  u_4(-s) \,ds}-\frac{\delta}{1+\alpha\int_{0}^{+\infty}f(s)  v_4(-s) ds}\Big|,\\
	&\leq [(2\beta + q_1+ q_2) N+ \mu + \rho +\delta]\| u-v\|_E \\
	& \ \ \ + (q_1+q_2+ \alpha \delta)N\| u-v\|_E=K_1\| u-v\|_E,
\end{align*}
where $K_1:=(2\beta +2 q_1+ 2q_2+\alpha \delta) N+\mu  +\rho + \delta $. In a similar way, we find that
\begin{equation*}
	\big| F_i(u)-F_i(v)\big|\leq K_i\| u-v\|_E, \qquad \text{for} \quad i=2,3,4,\end{equation*}
with $K_2:=2 q_2 N+\lambda +\mu  +\rho$,  $K_3:=(2\beta \sigma  + 2q_1 +\alpha \delta) N+\delta + \lambda +\mu$ and $K_4:=2\beta N (\sigma + 1) +\mu + \gamma$. Therefore, the system (\ref{sysmprinciple}) has a unique local solution associated to an initial condition
$\phi=(\phi_1,\phi_2,\phi_3,\phi_4)^T$ in $\Gamma$.

For a nonnegative initial condition, the positivity of $I$ is a direct consequence from the equation of $I$.
Indeed, we can write $I(t)=\phi_4(0) \exp(\int_{0}^{t}\beta S( y)+  \sigma
\beta Q(y) -(\gamma+ \mu) dy)\geq 0$. The following reasoning applies to $S$ and analogously to the other cases. Suppose that there exists $\Tilde{t}>0$ such that $S(t)\geq 0$, $V(t)\geq 0$, $Q(t)\geq 0$ for all $ 0\leq t\leq \Tilde{t}$ and $S(\Tilde{t})=0 $ with $S' (\Tilde{t})\leq 0$. We have from the equation of $S$ that $S'(\Tilde{t})=b +h\left(f* I\right) Q\left( \Tilde{t}\right) +\rho V\left( \Tilde{t}\right) \geq b > 0$, which leads to a contradiction. Thus, $S(t)\geq 0$ for all $t>0$. In a similar way, we show that $V(t)>0$ and $Q(t)>0$ for all
$t>0$.

By summing the fourth equations of system (\ref{sysmprinciple}), we get
$$(S+V+Q+I)'(t)=b-\mu(S+V+Q+I)(t)-\gamma I(t)\leq b-\mu(S+V+Q+I)(t). $$
Then, 
$$(S+V+Q+I)(t)\leq \Big(\sum\limits_{i=1}^{4}{\phi_i(0)}-\frac{b}{\mu}\Big)e^{-\mu t}+\frac{b}{\mu}\leq \frac{b}{\mu}. $$
This leads to the boundedness of solutions. We remark also that if $\phi \in \Gamma$ , then $u_t \in \Gamma$ for all $t>0$, then  $\Gamma$ is positively invariant. In summary, we state the following theorem.
\begin{theorem}
	For any nonnegative initial condition in $E$, system (\ref{sysmprinciple}) admits a unique, nonnegative, and bounded solution. Moreover, $\Gamma$ is positively invariant under system (\ref{sysmprinciple}).
\end{theorem}

\section{Basic reproduction number and steady states}\label{basexiseui}

In this section, we study the existence of equilibria. Before that, we computed the basic reproduction number of the corresponding ODE system without delay. We can apply the next generation matrix method proposed by Diekmann, Heesterbeek, and Metz \cite{r25}, and further developed by Van Den Driessche and Watmough \cite{r26,r27van}. Let's put $X=I $. From the system, we have
\begin{equation*}
	\frac{dX}{dt}=\digamma \left( X\right) -V\left( X\right),
\end{equation*}
where
$
\digamma \left( X\right) = \beta SI+\sigma \beta QI$ and $ V\left( X\right) =\left( \gamma +\mu \right) I.
$
The basic reproduction number, denoted $R_{0}$, is calculated using the next generation technique. At the disease-free equilibrium $E^{0}=\left( \frac{b}{\mu },0,0,0,0\right)$, we get
\begin{equation*}
	F=\left. \dfrac{\partial \digamma \left( X\right) }{\partial X}\right\vert
	_{E_{0}}=
	\dfrac{b\beta }{\mu } \qquad \text{and} \qquad
	V=\left. \frac{\partial V\left( X\right) }{\partial X}\right\vert
	_{E_{0}}=
	\left( \gamma +\mu \right).
\end{equation*}
Therefore, the basic reproduction number $R_{0}$ is the spectral radius of the next generation matrix, then
\begin{equation*}
	R_{0}=FV^{-1}=
	\dfrac{b\beta }{\mu \left( \gamma +\mu \right) }.
\end{equation*}
This threshold is used to determine the existence of steady states. Starting from the system \eqref{sysmprinciple} and using 
\begin{equation*}
	\int_{0}^{+\infty }f\left( a\right) da=1,
\end{equation*}
we consider a positive constant solution, denoted $\left( \overline{S},\overline{V},\overline{Q},\overline{I}\right) $, of the following system
\begin{equation}
	\left\{ 
	\begin{array}{l}
		b-\beta \overline{S}\overline{I}-q_{1}\overline{I}\overline{S}-q_{2}%
		\overline{I}\overline{S}-\mu \overline{S}+\frac{\delta }{1+\alpha \overline{I%
		}}\overline{Q}+\rho \overline{V}=0, \\ 
		\\ 
		q_{2}\overline{I}\overline{S}+\lambda \overline{Q}-\left( \rho +\mu \right) 
		\overline{V}=0, \\ 
		\\ 
		q_{1}\overline{I}\overline{S}-\frac{\delta }{1+\alpha \overline{I}}\overline{%
			Q}-\sigma \beta \overline{Q}\overline{I}-(\lambda +\mu )\overline{Q}%
		=0, \\ 
		\\ 
		\beta \overline{S}\overline{I}+\sigma \beta \overline{Q}\overline{I}-\left(
		\gamma +\mu \right) \overline{I}=0.
	\end{array}%
	\right.  \label{S2}
\end{equation}
According to the fourth equation of \eqref{S2}, we get
\begin{equation*}
	\left[ \beta \overline{S}+\sigma \beta \overline{Q}-\left( \gamma +\mu
	\right) \right] \overline{I}=0.
\end{equation*}
If we suppose that $\overline{I}=0$, we obtain the following equilibirum
\begin{equation*}
	E^{0}=\left( S^{0 },V^{0 },Q^{0 },I^{0 }\right)
	^{T}=\left( \frac{b}{\mu },0,0,0\right) ^{T}.
\end{equation*}
If we suppose that $\overline{I}>0$, then we have
\begin{equation*}
	\overline{S}=\left( \frac{b}{\mu R_{0}}-\sigma \overline{Q}\right).
\end{equation*}%
By replacing the expression of $\overline{S}$ in the third equation of \eqref{S2}, we obtain
\begin{equation*}
	\overline{Q}=\frac{bq_{1}\left( 1+\alpha \overline{I}\right) \overline{I}}{%
		\mu R_{0}\left[ \left( 1+\alpha \overline{I}\right) \left[ \sigma q_{1}%
		\overline{I}+\sigma \beta \overline{I}+(\lambda +\mu )\right] +\delta \right]
	}.
\end{equation*}
Then, by replacing the expression of $\overline{Q}$ in the expression of $\overline{S}$,
we find
\begin{equation*}
	\overline{S}=\frac{b\left[ \left( 1+\alpha \overline{I}\right) \left[ \sigma
		\beta \overline{I}+(\lambda +\mu )\right] +\delta \right] }{\mu R_{0}\left[
		\left( 1+\alpha \overline{I}\right) \left[ \sigma q_{1}\overline{I}+\sigma
		\beta \overline{I}+(\lambda +\mu )\right] +\delta \right] }.
\end{equation*}%
Substituting the expressions of $\overline{S}$ and $\overline{Q}$ in the second equation of \eqref{S2}, we obtain
\begin{equation*}
	\overline{V}=\frac{bq_{2}\left[ \left( 1+\alpha \overline{I}\right) \left( 
		\overline{I}\sigma \beta +(\lambda +\mu )\right) +\delta \right] \overline{I}%
		+\lambda q_{1}b\left( 1+\alpha \overline{I}\right) \overline{I}}{\mu
		R_{0}\left( \rho +\mu \right) \left[ \left( 1+\alpha \overline{I}\right) %
		\left[ \sigma q_{1}\overline{I}+\sigma \beta \overline{I}+(\lambda +\mu )%
		\right] +\delta \right] }.
\end{equation*}
Using the expressions of $\overline{S}$, $\overline{Q}$ and $\overline{V}$, the first equation of \eqref{S2} implies
\begin{eqnarray}
	&&\left( \beta +q_{1}+q_{2}\right) \left( 1-\frac{q_{2}\rho }{\left( \rho
		+\mu \right) \left( \beta +q_{1}+q_{2}\right) }\right) \overline{I}
	\notag\\
	&=&\mu \left( R_{0}-1\right) +\frac{q_{1}\left[ \left( \delta +\mu
		R_{0}\sigma \right) \left( \rho +\mu \right) +\rho \lambda \right] \overline{%
			I}}{\left( \rho +\mu \right) \left[ \left( 1+\alpha \overline{I}\right) %
		\left[ \sigma \beta \overline{I}+(\lambda +\mu )\right] +\delta \right] } \notag\\
	&&+%
	\frac{q_{1}\alpha \left[ \left( \rho +\mu \right) \mu R_{0}\sigma +\rho
		\lambda \right] \overline{I}^{2}}{\left( \rho +\mu \right) \left[ \left(
		1+\alpha \overline{I}\right) \left[ \sigma \beta \overline{I}+(\lambda +\mu )%
		\right] +\delta \right] }.  \label{eq1}
\end{eqnarray}
We rewrite (\ref{eq1}) as $A\left( \overline{I}\right) =0$, with
\begin{equation*}
	A\left( x\right) =d_{3}x^{3}+d_{2}x^{2}+d_{1}x+d_{0},
\end{equation*}
where
\begin{center}
	\[
	\left\{
	\begin{array}{r c l}
		d_{3}&=&-\alpha \sigma \beta \left[ \beta \left( \rho +\mu \right)
		+\allowbreak q_{1}\left( \rho +\mu \right) +\mu q_{2}\right],  \\ 
		\\
		d_{2}&=&\left( \rho +\mu \right) \left[ \alpha \sigma \mu q_{1}R_{0}+\alpha
		\sigma \beta \mu \left( R_{0}-1\right) -\alpha \beta \left( \lambda +\mu
		\right) -\alpha \mu q_{1}-\sigma \beta ^{2}-\sigma \allowbreak \beta q_{1}%
		\right]\\   
		&&-\alpha \mu q_{2}\left( \lambda +\mu \right) -\alpha \mu \lambda q_{1}-\mu
		\sigma \beta q_{2}, \\ 
		\\ 
		d_{1}&=&\left( \rho +\mu \right) \left[ \left( \sigma \allowbreak \beta \mu
		+\alpha \lambda \mu +\alpha \mu ^{2}\right) \left( R_{0}-1\right) +\sigma
		\mu q_{1}R_0-\beta \left( \lambda +\mu \right) -\mu q_{1}-\beta \delta \right]\\  
		&&-\mu q_{2}\left( \lambda +\mu \right) -\mu \lambda q_{1}-\delta q_{2}\mu,  \\ 
		\\ 
		d_{0}&=&\mu \left( \rho +\mu \right) \left( \mu +\lambda +\delta \right)
		\left( R_{0}-1\right).
	\end{array}
	\right.
	\]
\end{center}
Clearly, $\displaystyle\lim_{x\to +\infty }A(x)=-\infty $ and if $R_{0}> 1$, we have $A\left( 0\right) =d_{0}>0$. Hence, there exists  $\overline{I%
}\in \left( 0, +\infty \right) $ such that $A\left( \overline{I}\right) =0$. If $R_{0}\leq 1,$ then $d_{1}$,$d_{2}$,$d_{3}< 0$ and $A\left(0\right) =d_{0}\leq 0$. Then, $A$ is strictly decreasing on $\left( 0,+\infty
\right) $. Therefore, there is no positive roots $\overline{I}$ of $A$.

To further investigate the roots of the polynomial $A\left( \overline{I}\right) $, we can apply Descartes' Rule of Signs. The results are summarized in the table below.

\begin{equation*}
	\begin{array}{|c|c|c|c|c|c|}
		\hline
		d_{3} & d_{2} & d_{1} & d_{0} & \text{ Sign change} & \text{%
			Number of root} \\ 
		\hline
		- & + & + & + & 1 & 1 \\
		\hline
		- & - & + & + & 1 & 1 \\
		\hline
		- & + & - & + & 3 & 1, \ 2 \ \text{or} \ 3 \\ 
		\hline
		- & - & - & + & 1 & 1\\
		\hline
	\end{array}%
\end{equation*}

As a consequence, the system exhibits the possibility of multiple endemic equilibria depending on the parameter values, a feature that was not observed in the original model from \cite{kunireccuwave}. To illustrate this, we provide figures that depict the occurrence of either a single or multiple equilibria arising under different parameter values. The Figure \ref{EE1} shows the existence of a unique positive root of $A$ at $\overline{I}=0.8752$. The Figure \ref{EE2} shows the existence of two distinct positive real roots, a simple root at $\overline{I}=0.0192$ and a double root at $\overline{I}=0.6269$. The Figure \ref{EE3} shows the existence of three positive real roots at $\overline{I}=0.0199$, $\overline{I}=0.2470$ and $\overline{I}=1.0080$.
\begin{figure}[H]
	\begin{center}
		\includegraphics[width=10cm]{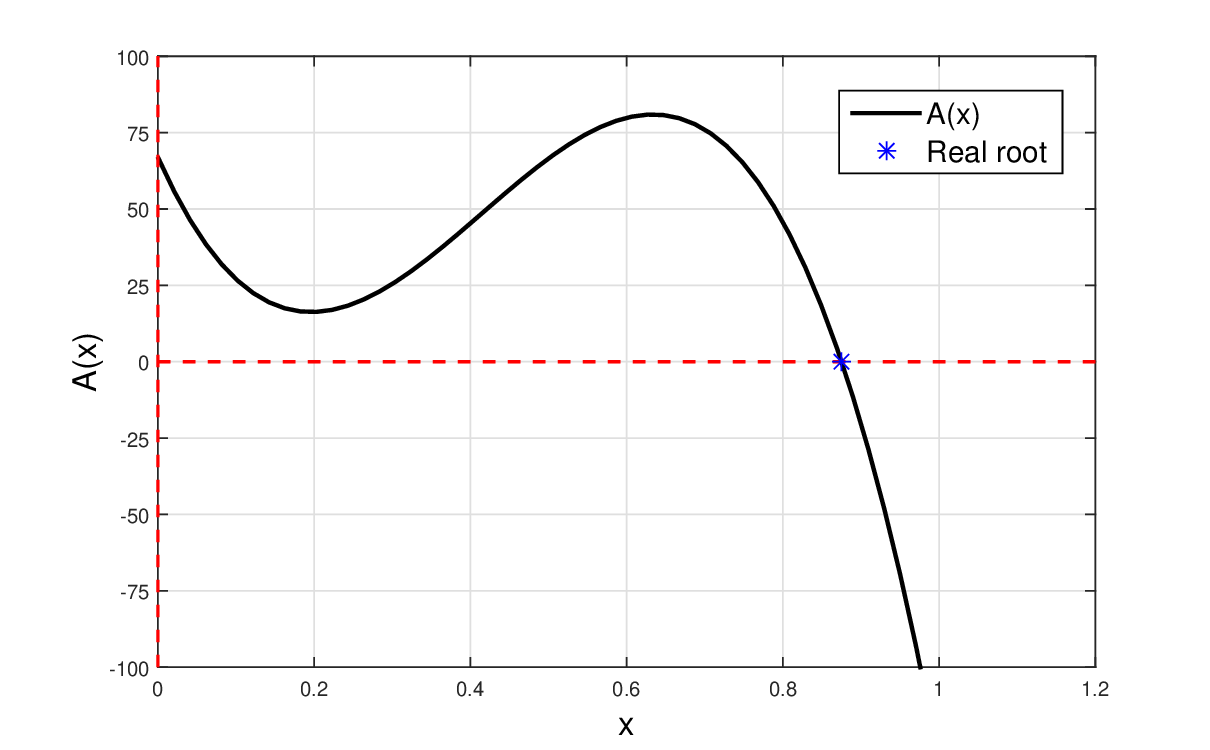}
		\caption{{\small The graph of the function $A$ shows that it has a single simple positive root for the given parameters: $\alpha = 130$, $\sigma = 0.5$, $\beta= 0.039$, $\mu= 0.045$, $q_1= 120$, $q_2= 90$, $\lambda= 0.002$, $\delta=140$, $\rho=5$,
				$b= 0.52$, $\gamma = 0.1$. The basic reproduction number is computed and we found $R_0=3.1080>1$.}}
		\label{EE1}
	\end{center}
\end{figure}

\begin{figure}[H]
	\begin{center}
		\includegraphics[width=10cm]{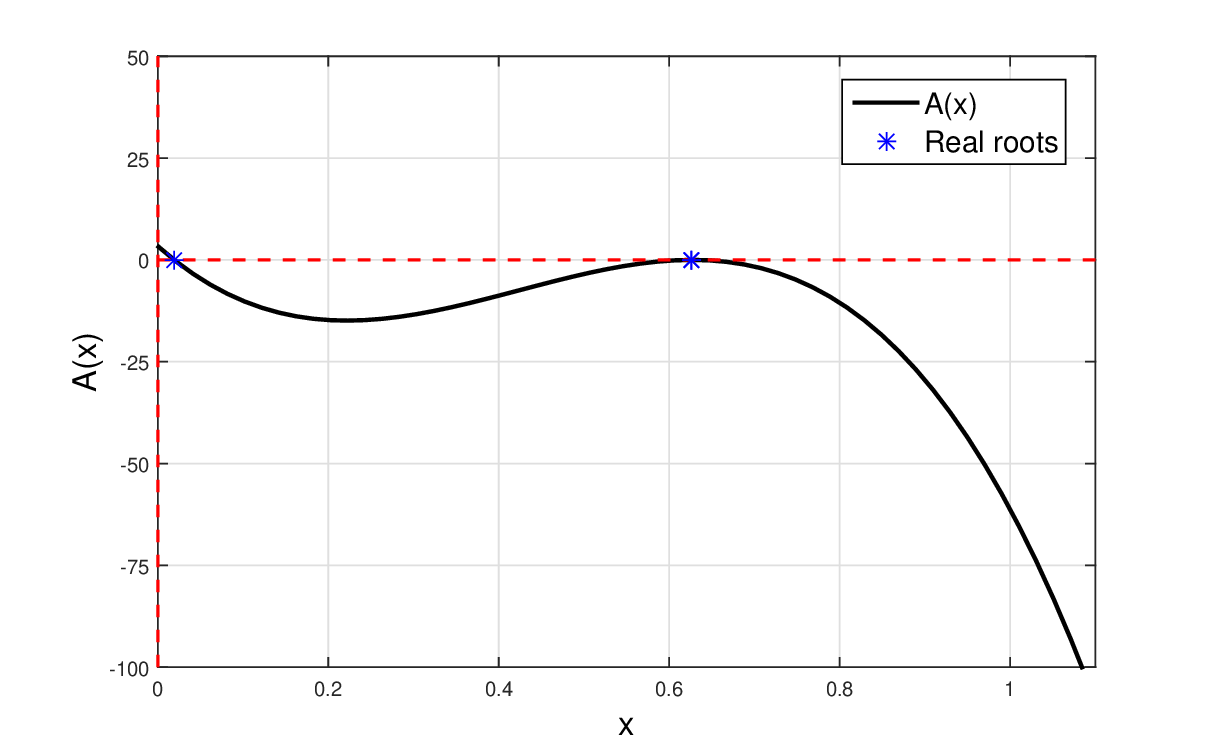}
		\caption{{\small The graph of the function $A$ shows that it has one double positive root and one simple positive root using the parameters: $\alpha = 177.46$, $\sigma = 0.5$, $\beta= 0.04$, $\mu= 0.045$, $q_1= 509.99856$, $q_2= 30$, $\lambda = 0.00200000043$, $\delta = 140.0001$, $\rho  = 0.2$, $b= 0.52$, $\gamma = 0.1$. We found $R_0=3.1877>1$.}}
		\label{EE2}
	\end{center}
\end{figure}

\begin{figure}[H]
	\begin{center}
		\includegraphics[width=10cm]{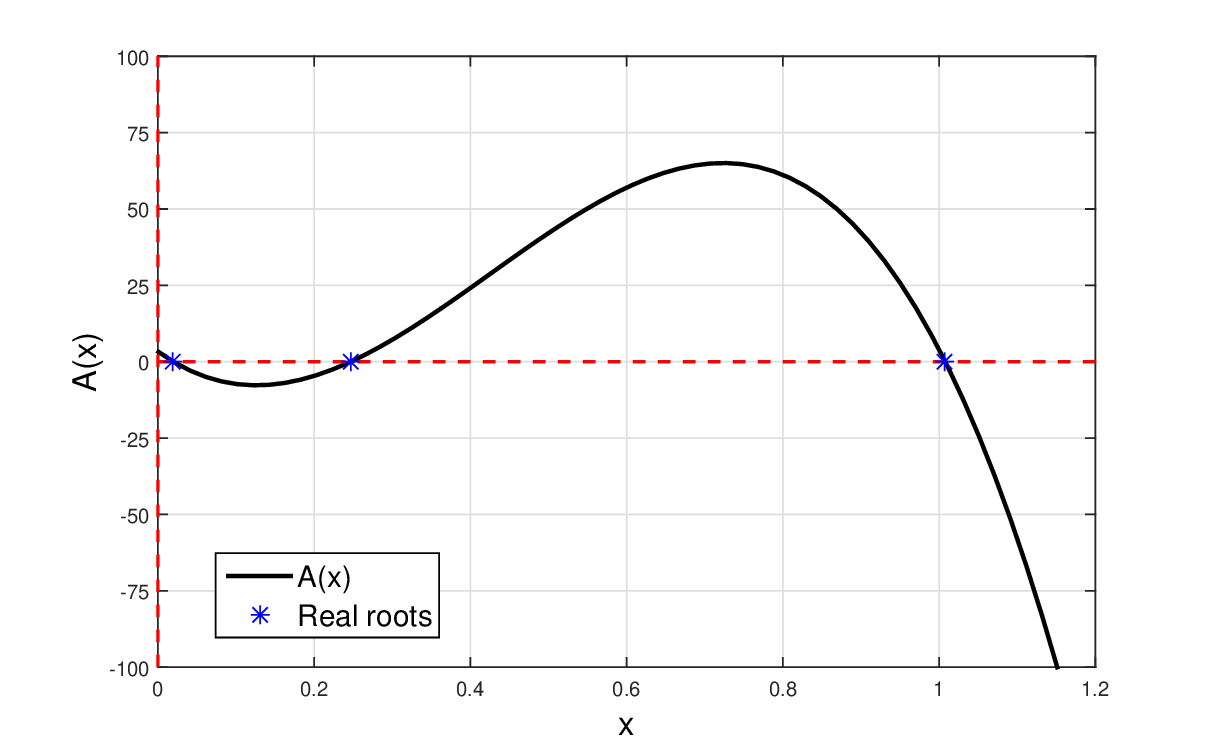}
		\caption{{\small The graph of the function $A$ shows the existence of three positive real roots for the given parameters: $\alpha = 270$, $\sigma = 0.5$, $\beta  = 0.04$, $\mu= 0.045$, $q_1= 510$, $q_2= 30$, $\lambda = 0.002$, $\delta = 140$, $\rho= 0.2$, $b= 0.52$, $\gamma = 0.1$. The basic reproduction number $R_0=3.1877>1$.}}
		\label{EE3}
	\end{center}
\end{figure}

We conclude this section by the following theorem.
\begin{theorem}
	The system \eqref{sysmprinciple} admits the following equilibrium points:
	\\
	1- A unique disease-free steady state $E^{0}$ always exists and it is given by
	\begin{equation*}
		E^{0}=\left( S^{0 },V^{0 },Q^{0 },I^{0 }\right)
		^{T}=\left( \frac{b}{\mu },0,0,0\right) ^{T}.
	\end{equation*}
	2- From a single to three endemic steady states $E^{*}=\left( S^{\ast },V^{\ast },Q^{\ast },I^{ \ast}\right) ^{T}$ exist if and only if $R_0>1$.
\end{theorem}

It is straightforward to observe that when $\sigma$ is sufficiently small, the model admits a unique endemic equilibrium. For instance, the condition $\sigma R_0 < 1$ guarantees uniqueness. If $\sigma$ is large (close to 1), then multiple endemic equilibria may arise. In such a scenario, multiple levels of infection may emerge in the system, reflecting the coexistence of distinct epidemic regimes depending on the initial conditions. This highlights a multistable dynamic, where the system can evolve toward different endemic equilibria.

Next, we will study the stability of equilibria locally and globally, as well as the existence of periodic solutions.

\section{Stability of the disease-free steady state}\label{ststeadyDFE}

To investigate the local stability of the equilibria, we linearize system \eqref{sysmprinciple} around an equilibrium point, namely the disease-free equilibrium (DFE) or an endemic equilibrium (EE). To do so, we define the following variables: $X=S-\overline{S},Y=V-\overline{V},Z=Q-\overline{Q},T=I-\overline{%
	I}$, we have
\begin{eqnarray*}
	g_{1}\left( f\ast I\right) S &=&q_{1}S\int_{0}^{+\infty }f\left( a\right)
	I\left( t-a\right) da, \\
	&=&q_{1}\left( X+\overline{S}\right) \int_{0}^{+\infty }f\left( a\right) %
	\left[ T\left( t-a\right) +\overline{I}\right] da, \\
	&=&q_{1}\overline{S}\overline{I}+q_{1}X\overline{I}
	+ q_{1}\overline{S}\int_{0}^{+\infty }f\left( a\right) T\left( t-a\right)
	da\\
	&&+q_{1}X\int_{0}^{+\infty }f\left( a\right) T\left( t-a\right) da,
\end{eqnarray*}
\begin{eqnarray*}
	g_{2}\left( f\ast I\right) S&=&q_{2}\overline{S}\overline{I}+q_{2}X\overline{I}%
	+q_{2}\overline{S}\int_{0}^{+\infty }f\left( a\right) T\left( t-a\right)
	da\\
	&&+q_{2}X\int_{0}^{+\infty }f\left( a\right) T\left( t-a\right) da,
\end{eqnarray*}
and
\begin{eqnarray*}
	h\left( I\right) Q &=&\frac{\delta Q}{1+\alpha \int_{0}^{+\infty }f\left(
		a\right) I\left( t-a\right) da}, \\
	&=&\frac{\delta \left( Z+\overline{Q}\right) }{1+\alpha
		\int_{0}^{+\infty }f\left( a\right) \left[ T\left( t-a\right) +\overline{I}%
		\right] da}, \\
	&=&\frac{\delta \left( Z+\overline{Q}\right) }{1+\alpha \overline{I}}\frac{1%
	}{1+\frac{\alpha \int_{0}^{+\infty }f\left( a\right) \left[ T\left(
			t-a\right) +\overline{I}\right] da}{1+\alpha \overline{I}}},\\
	&=&\frac{\delta \left( Z+\overline{Q}\right) }{1+\alpha 
		\overline{I}}\sum_{n=0}^{+\infty }\left( -1\right) ^{n}\left[ \frac{\alpha
		\int_{0}^{+\infty }f\left( a\right) \left[ T\left( t-a\right) +\overline{I}%
		\right] da}{1+\alpha \overline{I}}\right] ^{n}, \\
	&=&\frac{\delta \left( Z+\overline{Q}\right) }{\left( 1+\alpha \overline{I}%
		\right) }-\frac{\alpha \delta \overline{Q}}{\left( 1+\alpha \overline{I}%
		\right) ^{2}}\int_{0}^{+\infty }f\left( a\right) \left[
	T\left( t-a\right) \right]da\\
	&+&(\text{higher order terms}).
\end{eqnarray*}
We obtain to the following linearized system, for $t\geq 0$,
\begin{equation}
	\left\{ 
	\begin{array}{lll}
		X^{\prime }\left( t\right)  & = & -\left[ \left( \beta +q_{1}+q_{2}\right) 
		\overline{I}+\mu \right] X+\rho Y+\frac{\delta }{1+\alpha \overline{I}}%
		Z-\beta \overline{S}T \\ 
		&  & -\left[ q_{1}\overline{S}+q_{2}\overline{S}+\frac{\alpha \delta 
			\overline{Q}}{\left( 1+\alpha \overline{I}\right) ^{2}}\right]
		\int_{0}^{+\infty }f\left( a\right) T\left( t-a\right) da, \\
		\\ 
		Y^{\prime }\left( t\right)  & = & q_{2}\overline{I}X-\left( \rho +\mu
		\right) Y+\lambda Z+q_{2}\overline{S}\int_{0}^{+\infty }f\left( a\right)
		T\left( t-a\right) da, \\ 
		\\
		Z^{\prime }\left( t\right)  & = & q_{1}\overline{I}X-\left( \frac{\delta }{%
			1+\alpha \overline{I}}+\sigma \beta \overline{I}+\lambda +\mu \right)
		Z-\sigma \beta \overline{Q}T\\
		&&+\left[ q_{1}\overline{S}+\frac{\alpha \delta 
			\overline{Q}}{\left( 1+\alpha \overline{I}\right) ^{2}}\right]
		\int_{0}^{+\infty }f\left( a\right) T\left( t-a\right) da, \\
		\\ 
		T^{\prime }\left( t\right)  & = & \beta \overline{I}X+\sigma \beta \overline{%
			I}Z+\left[ \sigma \beta \overline{Q}+\beta \overline{S}-\left( \gamma +\mu
		\right) \right] T.%
	\end{array}%
	\right.   \label{S4}
\end{equation}
The following proposition states that the local stability of the
disease-free equilibrium is determined by the basic reproduction number $%
R_{0}$.
\begin{theorem}
	If $R_0<1$, the disease-free equilibrium $E^0$ of system (\ref{sysmprinciple}) is locally asymptotically stable in $\Gamma$. Conversely, if $R_0>1$, then $E^0$ becomes unstable.
\end{theorem}
\begin{proof}
	Based on system (\ref{S4}), the characteristic equation at the disease-free
	equilibrium $E^{0}$ is given, for $\zeta \in \mathbb{C}$, as follows
	{\footnotesize
		\begin{equation*}
			P\left( \zeta \right) =\left\vert 
			\begin{array}{cccc}
				\zeta +\mu  & -\rho  & -\delta  & \beta N+\left( q_{1}N+q_{2}N\right)
				\int_{0}^{+\infty }f\left( a\right) e^{-\zeta a}da \\ 
				0 & \zeta +\left( \rho +\mu \right)  & -\lambda  & -q_{2}N\int_{0}^{+\infty
				}f\left( a\right) e^{-\zeta a}da \\ 
				0 & 0 & \zeta +\left( \delta +\lambda +\mu \right)  & -q_{1}N\int_{0}^{+%
					\infty }f\left( a\right) e^{-\zeta a}da \\ 
				0 & 0 & 0 & \zeta -\left[ \beta N-\left( \gamma +\mu \right) \right] 
			\end{array}%
			\right\vert =0.
	\end{equation*}}
	Thus, the eigenvalues are given by
	\begin{equation*}
		\left\{ 
		\begin{array}{rcl}
			\zeta _{1} & = & -\mu <0, \\ 
			&  &  \\ 
			\zeta _{2} & = & -\left( \rho +\mu \right) <0, \\ 
			&  &  \\ 
			\zeta _{3} & = & -\left( \delta +\lambda +\mu \right) <0, \\ 
			&  &  \\ 
			\zeta _{4} & = & \beta N-\left( \gamma +\mu \right) =\left( \gamma +\mu
			\right) \left( R_{0}-1\right) <0.%
		\end{array}%
		\right. 
	\end{equation*}%
	Clearly, if $R_{0}<1$, the eigenvalues are negative, and then the
	disease-free equilibrium $E^{0}$ is locally asymptotically stable. On the other hand, if $R_{0}\geq 1$, then $\zeta _{4}> 0$ is positive, which implies that $E^{0}$ is unstable. This completes the proof.
\end{proof}

To complete the stability part, we investigate the global stability of the disease-free steady state. The main result is stated in the following theorem.
\begin{theorem}
	If $R_{0}\leq 1$, then the disease-free equilibrium of system \eqref{sysmprinciple} is globally asymptotically stable in $\Gamma$.
\end{theorem}
\begin{proof}
	We consider the function $W:\Gamma \rightarrow 
	\mathbb{R}^+
	$ defined by%
	\begin{equation*}
		W\left( S,V,Q,I\right) =\frac{1}{2}\left[ S\left( t\right)
		+V\left( t\right) +Q\left( t\right) +I\left( t\right)-\dfrac{b}{\mu} \right] ^{2}+%
		\frac{1 }{\beta }\left( 2\mu +\gamma \right) I\left( t\right).
	\end{equation*}
	The derivative of $W$, along the trajectory solution of system \eqref{sysmprinciple}, is given by
	\begin{equation*}
		W^{\prime }=\frac{\partial W}{\partial S}S^{\prime }\left( t\right) +\frac{%
			\partial W}{\partial V}V^{\prime }\left( t\right) +\frac{\partial W}{%
			\partial I}I^{\prime }\left( t\right) +\frac{\partial W}{\partial Q}%
		Q^{\prime }\left( t\right).
	\end{equation*}
	Then,
	\begin{equation*}
		W^{\prime }=\left[ S  +V
		+Q +I-\dfrac{b}{\mu}\right] \left( S^{\prime }+V^{\prime
		}+I^{\prime }+Q^{\prime }\right) +\frac{1 }{\beta }\left( 2\mu +\gamma
		\right) I^{\prime }.
	\end{equation*}
	Therefore, we obtain
	\begin{eqnarray*}
		W^{\prime } &=&\left[ S +V+Q+I-\dfrac{b}{\mu}\right] \left( b-\mu S-\mu
		V-\mu Q-\left( \gamma +\mu \right) I\right)  \\
		&&+\frac{1}{\beta }\left( 2\mu +\gamma \right) \left[ \beta S+\sigma
		\beta Q-\left( \gamma +\mu \right) \right] I, \\
		&=&-\mu \left[ S +V+Q-\dfrac{b}{\mu}\right] ^{2}-\mu \left[ \left(
		S-\dfrac{b}{\mu}\right) +V+Q\right] I \\
		&&-\left( \gamma +\mu \right) \left[ \left( S-\dfrac{b}{\mu}\right) +V+Q\right] I-\left(
		\gamma +\mu \right) I^{2} \\
		&&+\frac{1 }{\beta }\left( 2\mu +\gamma \right) \left[ \beta S+\sigma
		\beta Q-\left( \gamma +\mu \right) \right] I.
	\end{eqnarray*}
	Again, we have
	\begin{eqnarray*}
		W^{\prime } &=&-\mu \left[ S +V+Q-\dfrac{b}{\mu}\right] ^{2}-\left( \gamma
		+2\mu \right) \left[  S +V+Q-\dfrac{b}{\mu}\right] I-\left( \gamma +\mu
		\right) I^{2} \\
		&&+\frac{1 }{\beta }\left( 2\mu +\gamma \right) \left[ \beta S+\sigma
		\beta Q-\left( \gamma +\mu \right) \right] I ,\\
		&=&-\mu \left[ S +V+Q-\dfrac{b}{\mu}\right] ^{2}-\left( \gamma +\mu \right)
		I^{2}-\left( \gamma +2\mu \right) VI-\frac{\left( \gamma +2\mu \right) }{%
			\beta }\left[ \beta \left( S-\dfrac{b}{\mu}\right) +\beta Q\right] I \\
		&&+\frac{1 }{\beta }\left( 2\mu +\gamma \right) \left[ \beta S+\sigma
		\beta Q-\left( \gamma +\mu \right) \right] I, \\
		&=&-\mu \left[  S +V+Q-\dfrac{b}{\mu}\right] ^{2}-\left( \gamma +\mu \right)
		I^{2}-\left( \gamma +2\mu \right) VI \\
		&&+\frac{\left( \gamma +2\mu \right) }{\beta }\left[ \beta \dfrac{b}{\mu}-\left( \gamma
		+\mu \right) -\beta \left( 1-\sigma \right) Q\right] I, \\
		&=&-\mu \left[  S +V+Q-\dfrac{b}{\mu}\right] ^{2}-\left( \gamma +\mu \right)
		I^{2}-\left( \gamma +2\mu \right) \left[ V+\left( 1-\sigma \right) Q\right] I
		\\
		&&+\frac{\left( \gamma +2\mu \right) }{\beta }\left[ \beta \dfrac{b}{\mu}-\left( \gamma
		+\mu \right) \right] I, \\
		&=&-\mu \left[  S +V+Q-\dfrac{b}{\mu}\right] ^{2}-\left( \gamma +\mu \right)
		I^{2}-\left( \gamma +2\mu \right) \left[ V+\left( 1-\sigma \right) Q\right] I
		\\
		&&+\frac{\left( \gamma +2\mu \right) \left( \gamma +\mu \right) }{\beta }%
		\left[ R_{0}-1\right] I.
	\end{eqnarray*}
	Using the fact that $R_{0}\leq 1$, %
	\begin{equation}\label{einwpri}
		W^{\prime}\leq -\left( \gamma +\mu \right) I^{2}\leq 0,
	\end{equation}
	which means that $W$ is decreasing. In addition, we know that $W\geq 0$, so
	\begin{equation*}
		\lim_{t\to +\infty }W\left( S\left( t\right) ,V\left( t\right)
		,Q\left( t\right) ,I\left( t\right) \right) =\overline{W}%
		\geq 0.
	\end{equation*}
	If we integrate the inequality (\ref{einwpri}) over $\left(
	0,t\right) $, we get%
	\begin{equation*}
		\left( \gamma +\mu \right) \int_{0}^{t}I^{2}\left( s\right) ds\leq W\left(
		S\left( 0\right) ,V\left( 0\right) ,Q\left( 0\right),I\left( 0\right) 
		\right) -W\left( S\left( t\right) ,V\left( t\right)
		,Q\left( t\right),I\left( t\right)  \right).
	\end{equation*}
	As $t$ approaches $+\infty $, the limit in the left hand side of the above expression exists due to
	monotonicity, then
	\begin{equation*}
		\lim_{t\to +\infty }\int_{0}^{t}I^{2}\left( s\right) ds\leq 
		\frac{1}{\left( \gamma +\mu \right) }\left[ W\left( S\left( 0\right)
		,V\left( 0\right) ,I\left( 0\right) ,Q\left( 0\right) 
		\right) -\overline{W}\right].
	\end{equation*}
	From system (\ref{sysmprinciple}), we can conclude that $I^{\prime }$ is bounded.
	Consequently, $I$ is uniformly continuous. By applying Barbalat's lemma to the function
	$\int_{0}^{t}I^{2}\left( s\right) ds$, we get%
	\begin{equation*}
		\lim_{t\to +\infty }I^{2}\left( t\right) =0 \qquad \Rightarrow \qquad
		\lim_{t\to +\infty }I\left( t\right) =0.
	\end{equation*}
	This leads to the conclusion that 
	\begin{equation*}
		\left( S\left( t\right) ,V\left( t\right) ,I\left( t\right) ,Q\left(
		t\right) \right) \underset{t\to +\infty }{\to }%
		\left( b/\mu,0,0,0\right).
	\end{equation*}
\end{proof}


\section{Stability of the endemic steady state}\label{staendmiequ}

In the rest of our study, we will analyze the case of endemic equilibrium. Due to the complexity of the characteristic equation, we restrict our attention to certain straightforward cases.
\begin{proposition}
	Assume that $R_{0}>1$. If $q_{1}$ and $q_{2}$ are sufficiently small, then the unique endemic equilibrium $E^{*}$ of system (\ref{sysmprinciple}) is locally asymptotically stable.
\end{proposition}
\begin{proof}
	According to system (\ref{S4}), the characteristic equation at the endemic
	equilibrium $E^{*}$ is given by $P\left( \zeta \right) =0$ where%
	\[
	\scalebox{0.7}{$ P\left( \zeta \right) =\left( 
		\begin{array}{cccc}
			\zeta +\left[ \left( \beta +q_{1}+q_{2}\right) \overline{I}+\mu \right]  & 
			-\rho  & -\frac{\delta }{1+\alpha \overline{I}} & \beta \overline{S}+\left[
			q_{1}\overline{S}+q_{2}\overline{S}+\frac{\alpha \delta \overline{Q}}{\left(
				1+\alpha \overline{I}\right) ^{2}}\right] \int_{0}^{+\infty }f\left(
			a\right) e^{-\zeta a}da \\ 
			-q_{2}\overline{I} & \zeta +\left( \rho +\mu \right)  & -\lambda  & -q_{2}%
			\overline{S}\int_{0}^{+\infty }f\left( a\right) e^{-\zeta a}da \\ 
			-q_{1}\overline{I} & 0 & \zeta +\left[ \frac{\delta }{1+\alpha \overline{I}}%
			+\sigma \beta \overline{I}+\lambda +\mu \right]  & \sigma \beta \overline{Q}%
			-\left( q_{1}\overline{S}+\frac{\alpha \delta \overline{Q}}{\left( 1+\alpha 
				\overline{I}\right) ^{2}}\right) \int_{0}^{+\infty }f\left( a\right)
			e^{-\zeta a}da \\ 
			-\beta \overline{I} & 0 & -\sigma \beta \overline{I} & \zeta -\left[ \sigma
			\beta \overline{Q}+\beta \overline{S}-\left( \gamma +\mu \right) \right] 
		\end{array}
		\right)
		$}
	\]
	By continuity, it is sufficient to consider the case where $%
	q_{1}\to 0$ and $q_{2}\to 0$. We can see that
	\begin{equation*}
		\left( \overline{S},\overline{V},\overline{Q},\overline{I}\right)
		\to \left( \frac{b}{\mu R_{0}},0,0,\frac{\mu \left(
			R_{0}-1\right) }{\beta }\right).
	\end{equation*}
	In same time, we get
	\begin{equation*}
		P\left( \zeta \right) \underset{q_{1},q_{2}\to 0^{+}}{%
			\to }\left\vert 
		\begin{array}{cccc}
			\zeta +\beta \overline{I}+\mu  & -\rho  & -\frac{\delta }{1+\alpha \overline{%
					I}} & \beta \overline{S} \\ 
			0 & \zeta +\left( \rho +\mu \right)  & -\lambda  & 0 \\ 
			0 & 0 & \zeta +\left( \frac{\delta }{1+\alpha \overline{I}}+\sigma \beta 
			\overline{I}+\lambda +\mu \right)  & 0 \\ 
			-\beta \overline{I} & 0 & -\sigma \beta \overline{I} & \zeta 
		\end{array}%
		\right\vert.
	\end{equation*}
	Thus, $\zeta=-\left( \frac{\delta }{1+\alpha \overline{I}}+\sigma \beta 
	\overline{I}+\lambda +\mu \right) <0$ is a negative real eigenvalue and
	\begin{equation*}
		P\left( \zeta \right) =\zeta ^{3}+A_{2}\zeta ^{2}+A_{1}\zeta +A_{0},
	\end{equation*}
	where
	\begin{equation*}
		\left\{ 
		\begin{array}{rcl}
			A_{2} & = & \left( 2\mu +\rho +\mu \left( R_{0}-1\right) \right) , \\ 
			\text{\ }A_{1} & = & \left( \mu +\rho \right) \left( \mu +\mu \left(
			R_{0}-1\right) \right) +b\dfrac{\beta }{R_{0}}\left( R_{0}-1\right) , \\ 
			A_{0} & = & b\dfrac{\beta }{R_{0}}\left( \mu +\rho \right) \left(
			R_{0}-1\right) .%
		\end{array}%
		\right. 
	\end{equation*}%
	We compute the following quantity  
	\begin{eqnarray*}
		H_{2} &=&\left\vert 
		\begin{array}{cc}
			A_{2} & 1 \\ 
			A_{0} & A_{1}%
		\end{array}%
		\right\vert =A_{1}A_{2}-A_{0}, \\
		&=&\left( \left( \mu +\rho \right) \left[ \left( \mu +\mu \left(
		R_{0}-1\right) \right) +\left( \mu +\rho \right) \right] +b\dfrac{\beta }{%
			R_{0}}\left( R_{0}-1\right) \right) \left( \mu +\mu \left( R_{0}-1\right)
		\right) .
	\end{eqnarray*}
	It is evident that $
	A_{i}>0$ for $i=1,2,3$ and $ H_{2}> 0$ if $\ R_{0}>1.$ Therefore, using the Li\'{e}nard-Chipart criteria \cite{r28}, all the
	eigenvalues have negative real part. So, $E^{*}$ is
	locally asymptotically stable.
\end{proof}
\begin{theorem}
	Assume that $R_{0}>1$. If $\delta $ is sufficiently large and $q_{2}$ is sufficiently small, then the endemic equilibrium $E^{*}$ of system
	(\ref{sysmprinciple}) is locally asymptotically stable.
\end{theorem}
\begin{proof}
	For $\delta \to +\infty $ and $%
	q_{2}\to 0$, then we can see that
	\begin{equation*}
		\left( \overline{S},\overline{V},\overline{Q},\overline{I}\right)
		\longrightarrow \left( \frac{b}{\mu R_{0}},0,0,\frac{\mu \left(
			R_{0}-1\right) }{\beta }\right) .
	\end{equation*}
	When $%
	q_{2}\to 0$, the characteristic equation is given by
	\[
	\scalebox{0.7}{$
		\left\vert 
		\begin{array}{cccc}
			\zeta +\left[ \left( \beta +q_{1}\right) \overline{I}+\mu \right]  & -\rho 
			& -\frac{\delta }{1+\alpha \overline{I}} & \beta \overline{S}+q_{1}\overline{%
				S}\int_{0}^{+\infty }f\left( a\right) e^{-\zeta a}da \\ 
			0 & \zeta +\left( \rho +\mu \right)  & -\lambda  & 0 \\ 
			-q_{1}\overline{I} & 0 & \zeta +\left( \frac{\delta }{1+\alpha \overline{I}}%
			+\sigma \beta \overline{I}+\lambda +\mu \right)  & -q_{1}\overline{S}%
			\int_{0}^{+\infty }f\left( a\right) e^{-\zeta a}da \\ 
			-\beta \overline{I} & 0 & -\sigma \beta \overline{I} & \zeta 
		\end{array}%
		\right\vert =0.
		$
	}
	\]
	This is equavalent to 
	\[
	\scalebox{0.7}{$
		\left\vert 
		\begin{array}{cccc}
			\zeta +\left[ \left( \beta +q_{1}\right) \overline{I}+\mu \right]  & -\rho 
			& -\frac{1}{1+\alpha \overline{I}} & \beta \overline{S}+q_{1}\overline{S}%
			\int_{0}^{+\infty }f\left( a\right) e^{-\zeta a}da \\ 
			0 & \zeta +\left( \rho +\mu \right)  & -\frac{\lambda }{\delta } & 0 \\ 
			-q_{1}\overline{I} & 0 & \frac{\zeta +\sigma \beta \overline{I}+\lambda +\mu 
			}{\delta }+\frac{1}{1+\alpha \overline{I}} & -q_{1}\overline{S}%
			\int_{0}^{+\infty }f\left( a\right) e^{-\zeta a}da \\ 
			-\beta \overline{I} & 0 & \frac{-\sigma \beta \overline{I}}{\delta } & \zeta 
		\end{array}%
		\right\vert =0.
		$
	}
	\]
	When $\delta \to +\infty $, we get
	\begin{equation*}
		\left\vert 
		\begin{array}{cccc}
			\zeta +\left[ \left( \beta +q_{1}\right) \overline{I}+\mu \right]  & -\rho 
			& -\frac{1}{1+\alpha \overline{I}} & \beta \overline{S}+q_{1}S\int_{0}^{+%
				\infty }f\left( a\right) e^{-\zeta a}da \\ 
			0 & \zeta +\left( \rho +\mu \right)  & 0 & 0 \\ 
			-q_{1}\overline{I} & 0 & \frac{1}{1+\alpha \overline{I}} & -q_{1}\overline{S}%
			\int_{0}^{+\infty }f\left( a\right) e^{-\zeta a}da \\ 
			-\beta \overline{I} & 0 & 0 & \zeta 
		\end{array}%
		\right\vert 
	\end{equation*}
	\begin{equation*}
		=\left\vert 
		\begin{array}{cccc}
			\zeta +\left[ \beta \overline{I}+\mu \right]  & -\rho  & 0 & \beta \overline{%
				S} \\ 
			0 & \zeta +\left( \rho +\mu \right)  & 0 & 0 \\ 
			-q_{1}\overline{I} & 0 & \frac{1}{1+\alpha \overline{I}} & -q_{1}\overline{S}%
			\int_{0}^{+\infty }f\left( a\right) e^{-\zeta a}da \\ 
			-\beta \overline{I} & 0 & 0 & \zeta 
		\end{array}%
		\right\vert 
	\end{equation*}
	\begin{equation*}
		=\frac{1}{1+\alpha \overline{I}}\left\vert 
		\begin{array}{ccc}
			\zeta +\left( \beta \overline{I}+\mu \right)  & -\rho  & \beta \overline{S}
			\\ 
			0 & \zeta +\left( \rho +\mu \right)  & 0 \\ 
			-\beta \overline{I} & 0 & \zeta 
		\end{array}%
		\right\vert =0.
	\end{equation*}
	So,%
	\begin{equation*}
		P\left( \zeta \right) =\frac{1}{1+\alpha \overline{I}}\{\zeta
		^{3}+A_{2}\zeta ^{2}+A_{1}\zeta +A_{0}\},
	\end{equation*}
	with
	\begin{equation*}
		\left\{ 
		\begin{array}{rcl}
			A_{2}&=&2\mu +\rho +\mu \left( R_{0}-1\right) , \\ 
			A_{1}&=&\left( \mu +\rho \right) \left( \mu +\mu \left( R_{0}-1\right) \right)
			+b\beta \frac{\left( R_{0}-1\right) }{R_{0}}, \\ 
			A_{0}&=&b\beta \left( \mu +\rho \right) \frac{\left( R_{0}-1\right) }{R_{0}}.%
		\end{array}%
		\right. 
	\end{equation*}%
	We compute the following quantity 
	\begin{eqnarray*}
		H_{2} &=&\left\vert 
		\begin{array}{cc}
			A_{2} & 1 \\ 
			A_{0} & A_{1}%
		\end{array}%
		\right\vert =A_{1}A_{2}-A_{0}, \\
		&=&\left( \left( \mu +\rho \right) \left[ \left( \mu +\mu \left(
		R_{0}-1\right) \right) +\left( \mu +\rho \right) \right] +b\beta \frac{%
			\left( R_{0}-1\right) }{R_{0}}\right) \left( \mu +\mu \left( R_{0}-1\right)
		\right). 
	\end{eqnarray*}
	Clearly, we have $
	A_{i}>0$ for $i=1,2,3$ and $ H_{2}> 0$ if $\ R_{0}>1.$ Therefore, using the Li\'{e}nard-Chipart criteria \cite{r28}, all the
	eigenvalues have negative real part. So, $E^{*}$ is
	locally asymptotically stable.
\end{proof}

\section{Hopf bifurcation analysis}\label{hopfpart}

In this section, assuming $R_0 > 1$, we revisit the local stability analysis and investigate the possibility of a Hopf bifurcation, at least in the case where the endemic equilibrium $E^*$ is unique (i.e., when $\sigma R_0 < 1$, see Section \ref{basexiseui}). Although this analysis could, in principle, be applied to other equilibria when multiple endemic equilibria exist, our observations in the last section indicate that a Hopf bifurcation occurs only when the endemic equilibrium is unique. To simplify the calculations, we define 
\begin{equation*}
	\begin{array}{lll}
		k_{1}=\left[ \left( \beta +q_{1}+q_{2}\right) \overline{I}+\mu \right], &
		k_{2}=\dfrac{\delta }{1+\alpha \overline{I}}, \\ 
		\\ 
		k_{3}=q_{2}\overline{S}\displaystyle\int_{0}^{+\infty }f\left( a\right) e^{-\zeta a}da,& 
		k_{4}=\left[ q_{1}\overline{S}+\dfrac{\alpha \delta \overline{Q}}{\left(
			1+\alpha \overline{I}\right) ^{2}}\right] \displaystyle\int_{0}^{+\infty }f\left(
		a\right) e^{-\zeta a}da, \\ 
		\\ 
		k_{5}=\left( \dfrac{\delta }{1+\alpha \overline{I}}+\sigma \beta \overline{I}%
		+\lambda +\mu \right), &
		k_{6}=\rho +\mu .%
	\end{array}%
\end{equation*}
Then, the characteristic equation  is given as follows
\begin{equation*}
	P\left( \zeta \right) =\left\vert 
	\begin{array}{cccc}
		\zeta +k_{1} & -\rho & -k_{2} & \beta \overline{S}+k_{3}+k_{4} \\ 
		-q_{2}\overline{I} & \zeta +k_{6} & -\lambda & -k_{3} \\ 
		-q_{1}\overline{I} & 0 & \zeta +k_{5} & \sigma \beta \overline{Q}-k_{4} \\ 
		-\beta \overline{I} & 0 & -\sigma \beta \overline{I} & \zeta +\left[ 
		\overline{S}\beta +\overline{Q}\sigma \beta -\left( \gamma +\mu \right) %
		\right]%
	\end{array}%
	\right\vert =0.
\end{equation*}
After some calculations and algebraic manipulations, we obtain
\begin{equation}
	P\left( \zeta \right) =\zeta ^{4}+A_{3}\zeta ^{3}+A_{2}\zeta ^{2}+A_{1}\zeta
	+A_{0}+\left( B_{2}\zeta ^{2}+B_{1}\zeta +B_{0}\right) \int_{0}^{+\infty
	}f\left( a\right) e^{-\zeta a}da=0,  \label{PC1}
\end{equation}
where,
\begin{equation*} 
	\begin{array}{rcl}
		A_{3}&=&\left( k_{1}+k_{5}+k_{6}\right),  \vspace{0.2cm}
		\\ 
		A_{2}&=&k_{1}k_{5}+k_{5}k_{6}+k_{1}k_{6}-\rho q_{2}\overline{I\allowbreak }%
		-k_{2}q_{1}\overline{I}+\overline{S}\beta ^{2}\overline{I}+\overline{Q}\sigma ^{2}\beta ^{2}%
		\overline{I} \\  
		&=&k_{5}k_{6}+k_{1}\mu +\rho \left[ \left( \beta +q_{1}\right) \overline{I}%
		+\mu \right] +k_{1}\left( \sigma \beta \overline{I}+\lambda +\mu \right)  \\ 
		&&+\left[ \left( \beta +q_{2}\right) \overline{I}+\mu \right] k_{2}+\overline{S}\beta ^{2}%
		\overline{I}+\sigma ^{2}\beta ^{2}\overline{I}\overline{Q} , 
		\vspace{0.2cm}\\ 
		A_{1}&=&k_{1}k_{5}k_{6}-k_{2}k_{6}q_{1}\overline{I}+\beta ^{2}k_{5}\overline{S}%
		\overline{I}+\beta ^{2}k_{6}\overline{S}\overline{I}-\lambda \rho q_{1}%
		\overline{I}-\rho k_{5}q_{2}\overline{I} \\ 
		&&+\sigma \beta ^{2}k_{2}\overline{Q}\overline{I}+\allowbreak \sigma ^{2}\beta
		^{2}k_{1}\overline{Q}\overline{I}+\sigma ^{2}\beta ^{2}k_{6}\overline{I}%
		\overline{Q}+\sigma \beta ^{2}q_{1}\overline{S}\overline{I}^{2} \\  
		&=&k_{1}\mu \left( \sigma \beta \overline{I}+\mu \right) +\left[ \rho \left(
		\beta +q_{1}\right) \overline{I}+\mu \right] \left( \sigma \beta \overline{I}%
		+\mu \right) +k_{1}\lambda \mu +\lambda \rho \left( \beta \overline{I}+\mu
		\right)  \\ 
		&&+\beta ^{2}k_{5}\overline{S}\overline{I}+\beta ^{2}k_{6}\overline{S}%
		\overline{I}+\mu \left[ \left( \beta +q_{2}\right) \overline{I}+\mu \right]
		k_{2}+\rho \left( \beta \overline{I}+\mu \right) k_{2} \\ 
		&&+\sigma \beta ^{2}k_{2}\overline{Q}\overline{I}+\allowbreak \sigma ^{2}\beta
		^{2}k_{1}\overline{Q}\overline{I}+\sigma ^{2}\beta ^{2}k_{6}\overline{Q}%
		\overline{I}+\sigma \beta ^{2}q_{1}\overline{S}\overline{I}^{2}, 
		\vspace{0.2cm}\\ 
		A_{0}&=&\beta ^{2}k_{5}k_{6}\overline{S}\overline{I}+\sigma ^{2}\beta
		^{2}k_{1}\allowbreak k_{6}\overline{I}\overline{Q}+\sigma \beta
		^{2}k_{6}q_{1}\overline{S}\overline{I}^{2} \\ 
		&&+\sigma \beta ^{2}\lambda \rho \overline{I}\overline{Q}-\sigma ^{2}\beta
		^{2}\rho q_{2}\overline{Q}\overline{I}^{2}+\sigma \beta ^{2}k_{2}k_{6}%
		\overline{Q}\overline{I} \\  
		&=&\beta ^{2}k_{5}k_{6}\overline{S}\overline{I}+\sigma ^{2}\beta ^{2}\left(
		k_{1}\allowbreak \mu +\allowbreak \rho \left[ \left( \beta +q_{1}\right) 
		\overline{I}+\mu \right] \right) \overline{Q}\overline{I} \\ 
		&&+\sigma \beta ^{2}k_{6}q_{1}\overline{S}\overline{I}^{2}+\sigma \beta
		^{2}\lambda \rho \overline{Q}\overline{I}+\sigma \beta ^{2}k_{2}k_{6}%
		\overline{Q}\overline{I}, 
	\end{array}%
\end{equation*}
and
\begin{equation*} 
	\begin{array}{rcl}
		B_{2}&=&\beta q_{2}\overline{S}\overline{I}+\beta \left( 1-\sigma \right)
		\left( q_{1}\overline{S}+\frac{\alpha \delta \overline{Q}}{\left( 1+\alpha 
			\overline{I}\right) ^{2}}\right) \overline{I},  
		\vspace{0.2cm}\\ 
		B_{1}&=&\allowbreak \left( \beta k_{5}+\beta k_{6}+\sigma \beta q_{1}\overline{%
			I}-\beta k_{2}-\sigma \beta k_{1}-\sigma \beta \allowbreak k_{6}\right)
		\left( q_{1}\overline{S}+\frac{\alpha \delta \overline{Q}}{\left( 1+\alpha 
			\overline{I}\right) ^{2}}\right) \overline{I} \\ 
		&&+\left( \beta k_{5}+\beta k_{6}+\sigma \beta q_{1}\overline{I}-\allowbreak
		\beta \rho \right) q_{2}\overline{S}\overline{I} \\ 
		&=&\left[ \beta \lambda +\beta \mu \left( 1-\sigma \right) +\left( 1-\sigma
		\right) \beta k_{6}\right] \left( q_{1}\overline{S}+\frac{\alpha \delta 
			\overline{Q}}{\left( 1+\alpha \overline{I}\right) ^{2}}\right) \overline{I}
		\\ 
		&&+\left( \beta k_{5}+\beta \mu \right) q_{2}\overline{S}\overline{I}-\sigma
		\beta q_{2}\left( \frac{\alpha \delta \overline{Q}}{\left( 1+\alpha 
			\overline{I}\right) ^{2}}\right) \overline{I}^{2}, 
		\vspace{0.2cm}\\ 
		B_{0}&=&\left( \beta k_{5}k_{6}+\sigma \allowbreak \beta \rho q_{2}\overline{I}%
		+\sigma \beta k_{6}q_{1}\overline{I}-\beta k_{2}k_{6}-\beta \allowbreak
		\lambda \rho -\allowbreak \sigma \beta k_{1}k_{6}\right) \left( q_{1}%
		\overline{S}+\frac{\alpha \delta \overline{Q}}{\left( 1+\alpha \overline{I}%
			\right) ^{2}}\right) \overline{I} \\ 
		&&+\left( \beta k_{5}k_{6}+\sigma \beta k_{6}q_{1}\overline{I}-\beta \rho
		k_{5}-\sigma \beta \rho q_{1}\overline{I}\right) q_{2}\overline{S}\overline{I%
		} \\  
		&=&\left( \beta k_{6}\mu \left( 1-\allowbreak \sigma \right) +\beta \lambda
		\mu \right) \left( q_{1}\overline{S}+\frac{\alpha \delta \overline{Q}}{%
			\left( 1+\alpha \overline{I}\right) ^{2}}\right) \overline{I} \\ 
		&&+\beta k_{5}\mu q_{2}\overline{S}\overline{I}-\sigma \beta q_{2}\mu \left( 
		\frac{\alpha \delta \overline{Q}}{\left( 1+\alpha \overline{I}\right) ^{2}}%
		\right) \overline{I}^{2}.%
	\end{array}%
\end{equation*}
For the sake of simplicity, we use a uniform distribution kernel instead of $f$. We assume that
\begin{equation}
	f\left( a\right) =\left\{ 
	\begin{array}{ll}
		\dfrac{1}{L},\quad\text{ if \ }T<a<T+L, \vspace{0.2cm}\\ 
		0,\qquad\text{otherwise.}%
	\end{array}%
	\right.  \label{DS1}
\end{equation}
Here, $T\geq 0$ denotes the reporting delay, that is, the time between infection and official case notification. While $L\geq 0$ corresponds to the time window during which the behavior of individuals is influenced by past reported cases. The equation \eqref{PC1} leads to 
\begin{equation}
	P\left(\zeta \right) =\zeta ^{4}+A_{3}\zeta ^{3}+A_{2}\zeta
	^{2}+A_{1}\zeta +A_{0}+\left( B_{2}\zeta ^{2}+B_{1}\zeta +B_{0}\right) \dfrac{%
		1-e^{\zeta L}}{\zeta L}e^{-\zeta T}=0.  \label{PC2}
\end{equation}
Our objective is to analyze the sign of the real parts of the roots of equation (\ref{PC2}) to determine the stability of the endemic equilibrium.
Assume that $T$ and $L$ are small enough. We consider $T=L=0$, we get
\begin{equation}
	P\left( T,\zeta \right) =\zeta ^{4}+A_{3}\zeta ^{3}+\left(
	A_{2}+B_{2}\right) \zeta ^{2}+\left( A_{1}+B_{1}\right) \zeta +
	A_{0}+B_{0} =0.  \label{PC3}
\end{equation}
To establish the local asymptotic stability of the endemic equilibrium when $T=L=0$, we can impose the following two conditions
\\
1) The coefficients must satisfy the conditions,
\begin{equation*}
	A_{3},\left(
	A_{2}+B_{2}\right) , \left( A_{1}+B_{1}\right) , \left(
	A_{0}+B_{0}\right) >0.
\end{equation*}
2) Additionally, we suppose that
\begin{equation*}
	A_{3}\left[ \left( A_{1}+B_{1}\right) \left( A_{2}+B_{2}\right) -A_{3}\left(
	A_{0}+B_{0}\right) \right] -\left( A_{1}+B_{1}\right) ^{2}>0.
\end{equation*}
Now, we want to identify a pair of purely imaginary conjugate roots $\zeta =$
$\pm iw$, $\left( w\neq 0\right) $ of the characteristic equation $(\ref{PC2})$ for
the analysis of the Hopf bifurcation at the endemic equilibrium. We will vary $T$ from $0$ as a parameter of bifurcation (we can suppose that $L$ is small enough). Let consider the following function, for $x\geq 0$, 
\begin{equation}
	K\left( x\right) =x^{5}+c_{4}x^{4}+c_{3}x^{3}+c_{2}x^{2}+c_{1}x+\left(
	r_{2}x^{2}+r_{1}x+r_{0}\right) \left( 1-\cos \left( \sqrt{x}L\right) \right),
	\label{EQB1}
\end{equation}
with
\begin{eqnarray*}
	c_{4} =\left( A_{3}^{2}-2A_{2}\right) , \ \ c_{3}=\left( \allowbreak
	A_{2}^{2}+2A_{0}-2A_{1}A_{3}\right), \ \ c_{2}=\left(
	A_{1}^{2}-2A_{0}A_{2}\right), \ \ c_{1}=A_{0}^{2}, 
\end{eqnarray*}
and
\begin{eqnarray*}
	r_{2} =-\frac{2B_{2}^{2}}{L^{2}}, \ \ r_{1}=-\frac{2\left(
		B_{1}^{2}-2B_{0}B_{2}\right) }{L^{2}}, \ \ r_{0}=-\frac{2B_{0}^{2}}{L^{2}}.
\end{eqnarray*}
\begin{proposition}
	Assume $R_{0}>1$. Let $w>0$. There is an equivalence between:
	
	1) $K\left( w^{2}\right) =0$. 
	
	2) There exists $\ T^{\ast }>0$ such that the characteristic equation (\ref%
	{PC2}) admits a unique conjugate pair of purely imaginary roots $\pm iw$ if
	and only if
	
	\begin{equation*}
		T\in \left\{ T^{\ast }+\dfrac{2n\pi }{w}, \ \ n\in \mathbb{N}\right\} .
	\end{equation*}
\end{proposition}
\begin{proof}
	We consider $w$ $\in \mathbb{R}^{\ast }$ and we put $\zeta=iw$. The equation (\ref{PC2}) can be written as
	\begin{equation}
		P_{1}\left( w\right) =P_{2}\left( w\right) e^{-iwT},  \label{eqb}
	\end{equation}
	where
	\begin{equation*}
		\left\{ 
		\begin{array}{c}
			P_{1}\left( w\right) =w^{4}-iA_{3}w^{3}-A_{2}w^{2}+iA_{1}w+A_{0}, \\
			\\
			P_{2}\left( w\right) =\left( -B_{2}w^{2}+iB_{1}w+B_{0}\right) \dfrac{1-e^{iwL}%
			}{iwL}. %
		\end{array}%
		\right.
	\end{equation*}
	
	Since $\left\vert e^{-iwT}\right\vert =1$, a necessary condition for $w\neq 0$
	to be a solution of (\ref{eqb}) is $\left\vert P_{1}\left( w\right)
	\right\vert =\left\vert P_{2}\left( w\right) \right\vert $, i.e
	
	\begin{equation}
		\left( w^{4}-A_{2}w^{2}+A_{0}\right) ^{2}+\left( A_{1}w-A_{3}w^{3}\right)
		^{2}=\left( \left( B_{0}-B_{2}w^{2}\right) ^{2}+B_{1}^{2}w^{2}\right) \frac{%
			2\left( 1-\cos \left( wL\right) \right) }{w^{2}L^{2}} . \label{eqb2}
	\end{equation}
	Then, 
	\begin{eqnarray*}
		&&w^{10}+\left( A_{3}^{2}-2A_{2}\right) w^{8}+\left( \allowbreak
		A_{2}^{2}+2A_{0}-2A_{1}A_{3}\right) w^{6}+\left(
		A_{1}^{2}-2A_{0}A_{2}\right) w^{4}+w^{2}A_{0}^{2} \\
		&&-\frac{2}{L^{2}}\left( w^{4}B_{2}^{2}+\left( B_{1}^{2}-2B_{0}B_{2}\right)
		w^{2}+B_{0}^{2}\right) \left( 1-\cos \left( wL\right) \right)  =0.
	\end{eqnarray*}
	The previous equation can be rewritten as $K\left( x\right) =0$ with $%
	x=w^{2}$. Consequently, there exists $x>0$ such that $K\left( x\right) =0$ if and only
	if the equation $\left\vert P_{1}\left( w\right) \right\vert =\left\vert
	P_{2}\left( w\right) \right\vert $ has a pair of purely imaginary conjugate
	solutions $\zeta =$ $\pm i\sqrt{x}$.
	
	Now, if we suppose $\left\vert P_{1}\left( w\right) \right\vert =\left\vert
	P_{2}\left( w\right) \right\vert $. Obviously $P_{2}\left(
	w\right) \neq 0$. Since $e^{iwT}$ acts as a rotation of angle $wT$ in the
	complex plan, it is clear that there exists
	infinitely many values of $T$ such that
	\begin{equation*}
		e^{-iwT}=\dfrac{P_{1}\left( w\right) }{P_{2}\left( w\right) }.
	\end{equation*}
	From above, we get
	\begin{equation} 
		\begin{array}{ll}
			\cos (wT)=\text{Re} \left \{\dfrac{P_{1}(w)}{P_{2}(w)} \right \}=\text{Re} \left \{\dfrac{iwL\left[
				A_{2}w^{2}-w^{4}-A_{0}+i\left( A_{3}w^{3}-A_{1}w\right) \right] }{\left(
				B_{0}-B_{2}w^{2}+iB_{1}w\right) \left( 1-e^{-iwL}\right) }\right \}=\varepsilon
			_{1}\left( w\right)  \\ 
			\\ 
			\sin (wT)=-\text{Im} \left \{\dfrac{P_{1}(w)}{P_{2}(w)} \right \}=-\text{Im} \left\{\dfrac{iwL\left[ A_{2}w^{2}-w^{4}-A_{0}+i(A_{3}w^{3}-A_{1}w)\right]
			}{\left( B_{0}-B_{2}w^{2}+iB_{1}w\right) \left( 1-e^{-iwL}\right)} \right \}%
			=\varepsilon _{2}\left( w\right) 
		\end{array}%
		\label{cos}
	\end{equation}%
	Thus, 
	\begin{eqnarray*}
		&&wL \{\left( A_{2}w^{2}-w^{4}-A_{0}\right) \left[ B_{1}w\left( 1-\allowbreak
		\cos wL\right) +\left( B_{0}-B_{2}w^{2}\right) \sin wL\right]  \\
		\varepsilon _{1}\left( w\right)  &=&\dfrac{-\left( A_{3}w^{3}-A_{1}w\right) %
			\left[ \left( B_{0}-B_{2}w^{2}\right) \left( 1-\cos wL\right) -B_{1}w\sin wL%
			\right] \}}{2\left( \left( B_{0}-B_{2}w^{2}\right)
			^{2}+w^{2}B_{1}^{2}\right) \left( 1-\cos wL\right) }, \\
		&& \\
		&&wL\{\left( A_{2}w^{2}-w^{4}-A_{0}\right) \left[ \left(
		B_{0}-B_{2}w^{2}\right) \left( 1-\cos wL\right) -B_{1}w\sin wL\right]  \\
		\varepsilon _{2}\left( w\right)  &=&-\dfrac{\left( A_{3}w^{3}-A_{1}w\right) %
			\left[ B_{1}w\left( 1-\cos wL\right) +\left( B_{0}-B_{2}w^{2}\right) \sin wL%
			\right] \}}{2\left( \left( B_{0}-B_{2}w^{2}\right)
			^{2}+w^{2}B_{1}^{2}\right) \left( 1-\cos wL\right) }.
	\end{eqnarray*}%
	The problem (\ref{cos}) is well defined since $\dfrac{P_{1}\left( w\right) }{%
		P_{2}\left( w\right) }$ is on the unit disc. We get the associate value of $T$ by the following formula
	\begin{equation*}
		T=T^{\ast }\left( w\right) +\dfrac{2n\pi }{w},\quad n\in \mathbb{N},
	\end{equation*}%
	\bigskip where $T^{\ast }$ is given by, %
	\begin{equation*}
		T^{\ast }\left( w\right) =\left\{ 
		\begin{array}{c}
			\dfrac{arc\cos \varepsilon _{1}\left( w\right) }{w}, \qquad \varepsilon _{1}\left(
			w\right) \geq 0, \vspace{0.2cm}\\ 
			\dfrac{2\pi -arc\cos \varepsilon _{1}\left( w\right) }{w}, \quad \varepsilon
			_{1}\left( w\right) <0.%
		\end{array}%
		\right. 
	\end{equation*}%
\end{proof}
The next proposition characterizes the crossing direction of these roots through the imaginary axis. 
\begin{proposition}
	When $\zeta =i\sqrt{x^{\ast }}$ solves (\ref%
	{PC2}). If $K'\left( x^{\ast }\right) >0$ (respectively $K'\left(
	x^{\ast }\right) $ $<0$), then as $T$ increases, the pair of complex conjugate roots $\pm i\sqrt{%
		x^{\ast }}$ crosses the imaginary axis from left to right (respectively right to left). 
\end{proposition}
\begin{proof} Recall the characteristic
	equation 
	\begin{equation}
		\zeta ^{5}+A_{3}\zeta ^{4}+A_{2}\zeta ^{3}+A_{1}\zeta ^{2}+A_{0}\zeta
		+\left( B_{2}\zeta ^{2}+B_{1}\zeta +B_{0}\right) \dfrac{1-e^{-\zeta L}}{L}%
		e^{-\zeta T}=0.  \label{PCQ3}
	\end{equation}
	We consider $\zeta =\zeta \left( T\right) $ and we have
	\begin{equation*}
		\frac{\partial }{\partial \zeta }P\left( T,\zeta \left( T\right) \right) 
		\frac{\partial }{\partial T}\zeta \left( T\right) +\frac{\partial }{\partial
			T}P\left( T,\zeta \left( T\right) \right) =0.
	\end{equation*}
	Then,
	\begin{eqnarray*}
		&\left( 5\zeta ^{4}+4A_{3}\zeta ^{3}+3A_{2}\zeta ^{2}+2A_{1}\zeta
		+A_{0}\right) \zeta ^{\prime } \\
		&+\left[ \left( 2B_{2}\zeta +B_{1}\right) \dfrac{1-e^{-\zeta L}}{L}+\left(
		B_{2}\zeta ^{2}+B_{1}\zeta +B_{0}\right) e^{-\zeta L}-T\left( B_{2}\zeta
		^{2}+B_{1}\zeta +B_{0}\right) \dfrac{1-e^{\zeta L}}{L}\right] e^{-\zeta
			T}\zeta ^{\prime } \\
		&-\zeta \left( B_{2}\zeta ^{2}+B_{1}\zeta +B_{0}\right) \dfrac{1-e^{-\zeta L}%
		}{L}e^{-\zeta T} =0.
	\end{eqnarray*}
	This gives us
	\begin{eqnarray*}
		\frac{1}{\zeta ^{\prime }} &=&\frac{1}{\zeta }\left[ \frac{\left( 5\zeta
			^{4}+4A_{3}\zeta ^{3}+3A_{2}\zeta ^{2}+2A_{1}\zeta +A_{0}\right) }{\left(
			B_{2}\zeta ^{2}+B_{1}\zeta +B_{0}\right) \dfrac{1-e^{-\zeta L}}{L}e^{-\zeta T}%
		}+\dfrac{\left( 2B_{2}\zeta +B_{1}\right) }{\left( B_{2}\zeta ^{2}+B_{1}\zeta
			+B_{0}\right) }+\dfrac{Le^{-\zeta L}}{1-e^{-\zeta L}}-T\right], \\
		&=&\dfrac{1}{\zeta }\left[ -\dfrac{\left( 5\zeta ^{4}+4A_{3}\zeta
			^{3}+3A_{2}\zeta ^{2}+2A_{1}\zeta +A_{0}\right) }{\zeta ^{5}+A_{3}\zeta
			^{4}+A_{2}\zeta ^{3}+A_{1}\zeta ^{2}+A_{0}\zeta }+\dfrac{\left( 2B_{2}\zeta
			+B_{1}\right) }{\left( B_{2}\zeta ^{2}+B_{1}\zeta +B_{0}\right) }+\dfrac{L}{%
			1-e^{-\zeta L}}-\left( T+L\right) \right].
	\end{eqnarray*}
	In order to identify the crossing direction of $\zeta \left( T\right) $  through the imaginary axis, we evaluate the sign of  $\text{Re}\left( \dfrac{1}{\zeta
		^{\prime }}\right) $ at $\zeta =iw$. For $\zeta =iw$, we get
	\begin{equation*}
		\frac{1}{\zeta ^{\prime }}=\frac{1}{w}\left[ 
		\begin{array}{c}
			\dfrac{5w^{4}-3A_{2}w^{2}+A_{0}+i\left( 2A_{1}w-4A_{3}w^{3}\right) }{%
				w^{5}-A_{2}w^{3}+A_{0}w+i\left( A_{1}w^{2}-A_{3}w^{4}\right) }-\dfrac{\left(
				i2B_{2}w+B_{1}\right) }{B_{1}w+i\left( B_{2}w^{2}-B_{0}\right) } \\ 
			\\ 
			-\dfrac{L}{\sin \left( wL\right) -i\left( 1-\cos wL\right) }+i\left(
			T+L\right) 
		\end{array}%
		\right]. 
	\end{equation*}
	So,
	\begin{equation*}
		\text{Re}\left( \dfrac{1}{\zeta ^{\prime }}\right) =\dfrac{1}{w}\left[ 
		\begin{array}{c}
			\dfrac{\left( 5w^{4}-3A_{2}w^{2}+A_{0}\right) \left(
				w^{5}-A_{2}w^{3}+A_{0}w\right) +\left( 2A_{1}w-4A_{3}w^{3}\right) \left(
				A_{1}w^{2}-A_{3}w^{4}\right) }{\left( w^{5}-A_{2}w^{3}+A_{0}w\right)
				^{2}+\left( A_{1}w^{2}-A_{3}w^{4}\right) ^{2}} \\ 
			\\ 
			-\dfrac{B_{1}^{2}w+2B_{2}w\left( B_{2}w^{2}-B_{0}\right) }{%
				B_{1}^{2}w^{2}+\left( B_{2}w^{2}-B_{0}\right) ^{2}}-\dfrac{L\sin \left(
				wL\right) }{2\left( 1-\cos Lw\right) }%
		\end{array}%
		\right], 
	\end{equation*}
	
	\begin{equation*}
		\text{Re}\left( \dfrac{1}{\zeta ^{\prime }}\right) =\dfrac{1}{w}\left[ 
		\begin{array}{c}
			\dfrac{5w^{9}+4c_{4}w^{7}+3\allowbreak c_{3}w^{5}+2c_{2}w^{3}+wc_{1}}{w^{2}%
				\left[ \left( w^{4}-A_{2}w^{2}+A_{0}\right) ^{2}+\left(
				A_{1}w-A_{3}w^{3}\right) ^{2}\right] } \\ 
			\\ 
			-\left[ \dfrac{2\left( 1-\cos Lw\right) \left[ B_{1}^{2}w+2B_{2}w\left(
				B_{2}w^{2}-B_{0}\right) \right] +L\sin \left( wL\right) \left[
				B_{1}^{2}w^{2}+\left( B_{2}w^{2}-B_{0}\right) ^{2}\right] }{2\left( 1-\cos
				Lw\right) \left[ B_{1}^{2}w^{2}+\left( B_{2}w^{2}-B_{0}\right) ^{2}\right] }%
			\right] 
		\end{array}%
		\right]. 
	\end{equation*}%
	Using (\ref{eqb2}), we get
	\[
	\scalebox{0.9}{$
		\begin{aligned}
			Re\left( \dfrac{1}{\zeta'} \right) 
			&= \dfrac{%
				5w^{8}+4c_{4}w^{6}+3c_{3}w^{4}+2c_{2}w^{2}+c_{1}}{w^{2}\left[
				\left( w^{4}-A_{2}w^{2}+A_{0}\right) ^{2}+\left( A_{1}w-A_{3}w^{3}\right)
				^{2}\right] } \\[1em]
			&\quad + \dfrac{\left( 1-\cos(Lw) \right) \left( 2r_{2}w^{2}+r_{1} \right) 
				+\left( r_{2}w^{4}+r_{1}w^{2}+r_{0} \right) \dfrac{L\sin(wL)}{2w}}
			{ \left[ B_{1}^{2}w^{2}+\left( B_{2}w^{2}-B_{0} \right) ^{2} \right] 
				\cdot \dfrac{2(1-\cos(Lw))}{L^{2}} } \\[1em]
			&= \frac{
				5w^8 + 4c_4 w^6 + 3c_3 w^4 + 2c_2 w^2 + c_1 
				+ (1 - \cos(Lw))(2r_2 w^2 + r_1)
				+ (r_2 w^4 + r_1 w^2 + r_0) \left( \dfrac{L \sin(Lw)}{2w} \right)
			}{
				w^2 \left[ (w^4 - A_2 w^2 + A_0)^2 + (A_1 w - A_3 w^3)^2 \right]
			} \\[1em]
			&= \dfrac{K'(w^2)}{
				w^{2}\left[ \left( w^{4}-A_{2}w^{2}+A_{0} \right) ^{2}+\left( A_{1}w-A_{3}w^{3} \right) ^{2} \right]
			}.
		\end{aligned}
		$}
	\]
	The proof is completed.
\end{proof}

\section{Numerical simulations}\label{simu}

In this section, we give a complete illustration of the results obtained previously. The parameters values are given bellow, in particular for the parameters inherited from \cite{kunireccuwave}, we have kept the same values. The total population $N=b/\mu =1$ for the values $b=\mu =%
\frac{1}{80}$ years$^{-1}$. The average life period is $\frac{1}{\mu }=80$ years. The average infection period is $\frac{1}{\gamma }=\frac{1}{24}$ years. The vaccination rate is $\lambda =0.1$ years$^{-1}$. The loss of vaccination immunity rate $\rho =0.1$ years$^{-1}$. The basic reproduction number $R_{0}$ is $2.5$. The efficacy of quarantine is $\left( 1-\sigma \right) \times $ $%
100\%$ $=70\%$. The average period of quarantine when $\alpha =0$ is $\frac{1}{\delta }=\frac{1}{12}$ years. The validity period of the information is $L=\frac{1}{4}$ years ($=3$ months). Additionally, we set the sensitivity of quarantine to $q_{1}$ $%
=75$, the sensitivity of vaccination and release to  $q_{2}=10$, and $\alpha
=1$. 

Under these conditions, we obtain an endemic equilibrium with $\overline{I}\approx 3.0605\times 10^{-4}$ and
$A_{3}\approx 12.2837,A_{2}\approx
2.2237,A_{1}= 5.4399,A_{0}=0.6016,B_{2}\approx  0.4592,B_{1}\approx
0.9940,B_{0}\approx  0.0124,$ $c_{4}\approx 146.4425,c_{3}\approx -127.4972,c_{2}\approx 26.9173,c_{1}\approx 0.3619$.

We numerically obtain the positive roots for equation $K\left(x\right) =0$. We find $x_{-}^{\ast }=0.3599$ and $x_{+}^{\ast }=0.5212,$ such that $k^{\prime }\left( x_{-}^{\ast }\right) <0$ and $k^{\prime }\left(x_{+}^{\ast }\right) >0$ (see Figure \ref{Kx}). So, $\omega _{-}=\sqrt{%
	x_{-}^{\ast }}$ and $\omega _{+}=\sqrt{x_{+}^{\ast }}$.
\begin{figure}[H]
	\begin{center}
		\includegraphics[width=10cm]{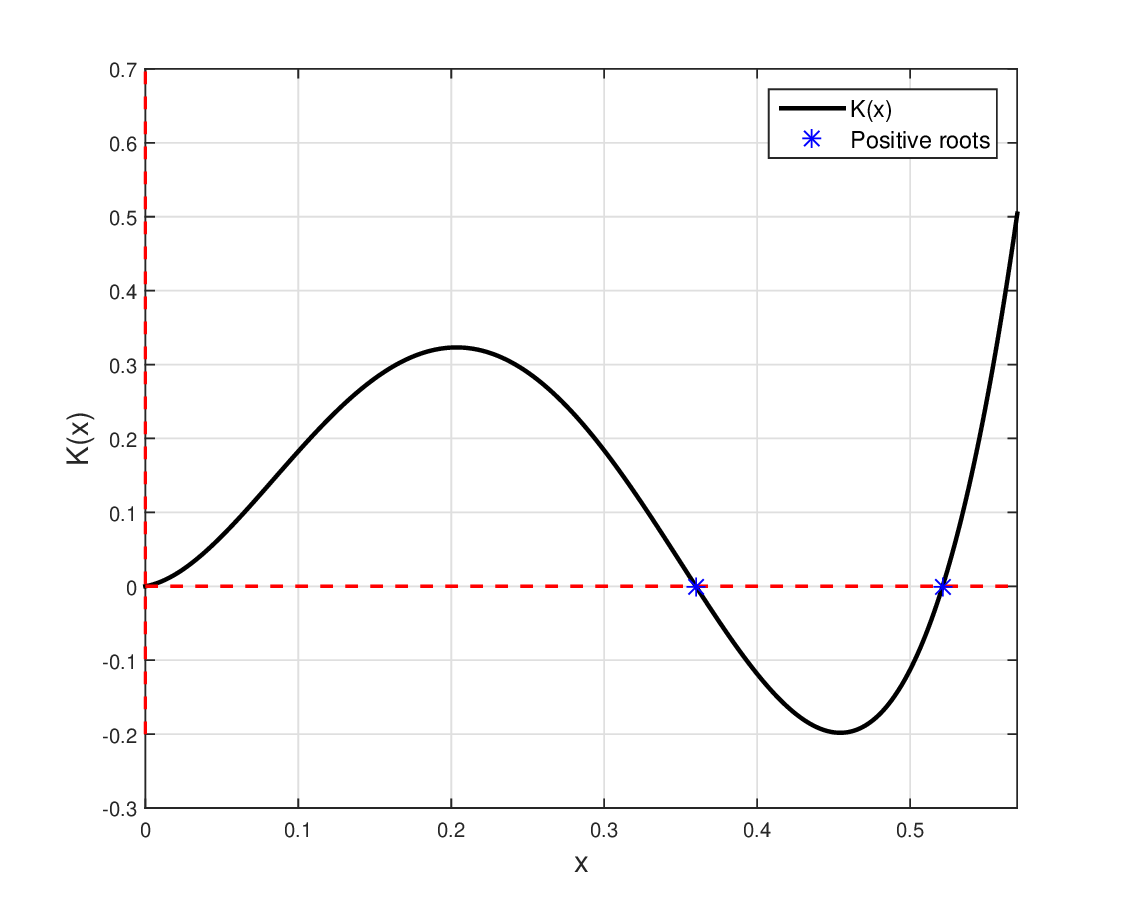}
		\caption{{\small The graph of the equation $K\left( x\right) =x^{5}+c_{4}x^{4}+c_{3}x^{3}+c_{2}x^{2}+c_{1}x+\left(
				r_{2}x^{2}+r_{1}x+r_{0}\right) \left( 1-\cos \left( \sqrt{x}L\right) \right)=0.$ We observe that $x_{-}^{\ast }=0.3599$ and $x_{+}^{\ast }=0.5212$. }}
		\label{Kx}
	\end{center}
\end{figure}
A conjugate pair $\zeta  =\pm iw_+$ (resp., $\zeta =\pm iw_-$) of purely imaginary roots of the characteristic equation \eqref{PC2} crosses the imaginary axis from left to right (respectively, from right to left) as the parameter $T$ passes through the critical value $T_{+}^{n}$ (resp., $T_{-}^{n}$). This
implies that the endemic equilibrium may lose stability (respectively, gain stability), and a nontrivial periodic solution may emerge (respectively, vanish) via a Hopf bifurcation at 
\begin{eqnarray*}
	T_{0}^{+} =0.8409,\text{ \ \ \ \ }T_{0}^{-}=5.3336,\text{ \ \ \ \ \ \ \ }%
	T_{1}^{+}=9.5441,\text{ \ \ \ \ \ \ \ \ }T_{1}^{-}=15.8071, \\
	T_{2}^{+} =18.2473,\text{ \ \ \ \ }T_{2}^{-}=26.2805,\text{ \ \ \ \ \ \ \ }%
	T_{3}^{+}=26.9504,\text{ \ \ \ \ \ \ \ \ }T_{3}^{-}=36.7539.
\end{eqnarray*}
Therefore, we can say that the endemic equilibrium is stable if $T\in S=\left( 0,T_{0}^{+}\right)
\cup \left( T_{0}^{-},T_{1}^{+}\right) \cup \left(
T_{1}^{-},T_{2}^{+}\right) \cup \left( T_{2}^{-},T_{3}^{+}\right) ,$  and it is unstable if
$T\in U=\left( T_{0}^{-},T_{0}^{+}\right) \cup \left( T_{1}^{-},T_{1}^{+}\right)
\cup \left( T_{2}^{-},T_{2}^{+}\right) \cup \left(T_{3}^{-},T_{3}^{+}\right) .$ The bifurcation at $%
T=T_{n}^{+}$, with $\ n=0,1,2,3,...$, is of a supercritical nature and at $T=T_{n}^{-}$, with $n=0,1,2,3,...$, is of a subcritical
nature.
{\color{black}
To simulate the solutions of system \eqref{sysmprinciple}, we perform a uniform
discretization of the time interval $\ \left[ 0,\overline{T}\right] $ as
follows
\begin{equation*}
0=t_{0}<t_{1}<\text{\textperiodcentered \textperiodcentered
\textperiodcentered }<t_{i-1}<t_{i}<t_{i+1}<\text{\textperiodcentered \textperiodcentered
\textperiodcentered }<t_{N}=\overline{T}.
\end{equation*}
We put $
\Delta t=\overline{T}/N$ and $t_{i}=i\Delta t$, where   $\Delta t$ represents the time step of the subdivision, and the
points $t_i$, for $i=0,\ldots ,N$, are the nodes of the discretization. The interval $[0,T+L]$ is also divided into $N_\tau $ subintervals such
that
\begin{equation*}
0=\tau _{0}<\tau _{1}<\text{\textperiodcentered \textperiodcentered
\textperiodcentered }<\tau _{j-1}<\tau _{j}<\tau _{j+1}<\text{%
\textperiodcentered \textperiodcentered \textperiodcentered }<\tau _{N_
\tau }=T+L.
\end{equation*}
For $j=0,\ldots,N_\tau $, we have
\begin{equation*}
N_\tau  =\frac{T+L}{\Delta t} \quad\text{ and } \quad\tau_j=j \Delta t.
\end{equation*}
We set
\begin{equation*}
S_{i}:=S(t_{i}),\text{ }Q_{i}:=Q(t_{i}),\text{ }I_{i}:=I(t_{i}),\text{ }%
V_{i}:=V(t_{i}),\text{ }I_{i-j}:=I(t_{i}-\tau _{j}).
\end{equation*}

To approximate the derivative in system \eqref{sysmprinciple}, we use the explicit Euler
scheme
\begin{equation*}
\left\{ 
\begin{array}{c}
S^{\prime }(t_{i})\approx \dfrac{S_{i+1}-S_{i}}{\Delta t}+O\left( \Delta t
\right),  \vspace{0.2cm}\\ 
Q^{\prime }(t_{i})\approx \dfrac{Q_{i+1}-Q_{i}}{\Delta t}+O\left( \Delta t
\right),   \vspace{0.2cm}\\ 
I^{\prime }(t_{i})\approx \dfrac{I_{i+1}-I_{i}}{\Delta t}+O\left( \Delta t
\right),   \vspace{0.2cm}\\ 
V^{\prime }(t_{i})\approx \dfrac{V_{i+1}-V_{i}}{\Delta t}+O\left( \Delta t
\right) .
\end{array}%
\right. 
\end{equation*}
We estimate the term $\int_{0}^{+\infty }f\left( \tau \right) I\left( t-\tau
\right) d\tau $ using the rectangle method. This method consists of
constructing a sum, called a Riemann sum, by using $N_\tau $ rectangles
with base $[\tau _{j},\tau _{j+1}]$ and height  $f(j\Delta t)I(t_{i}-%
\tau _{j})$. The corresponding approximation is given by
\begin{equation*}
\int_{0}^{+\infty }f\left( \tau \right) I\left( t-\tau \right) d\tau \approx
\sum_{j=0}^{N_{\tau}-1}f(j \Delta t )I(t_{i}-\tau_{j}),\qquad\text{ 
for }i=1,\ldots,N.
\end{equation*}
}
First, we graphically illustrate the global asymptotic stability of the disease-free equilibrium when $R_{0}<1$. The curves in Figure $\ref{staDFE}$ show how each disease-related state evolves over time. The susceptible population first decreases and then increases until stabilization. Vaccinated and quarantined individuals initially rise, then decline to zero. The infected population decreases immediately and approaches zero. These patterns indicate that the system is asymptotically stable around the disease-free equilibrium $E^0$, suggesting that that the disease can be eradicated under the given parameters when $R_{0}<1$.
\begin{figure}[H]
	\begin{center}
		\includegraphics[width=14cm]{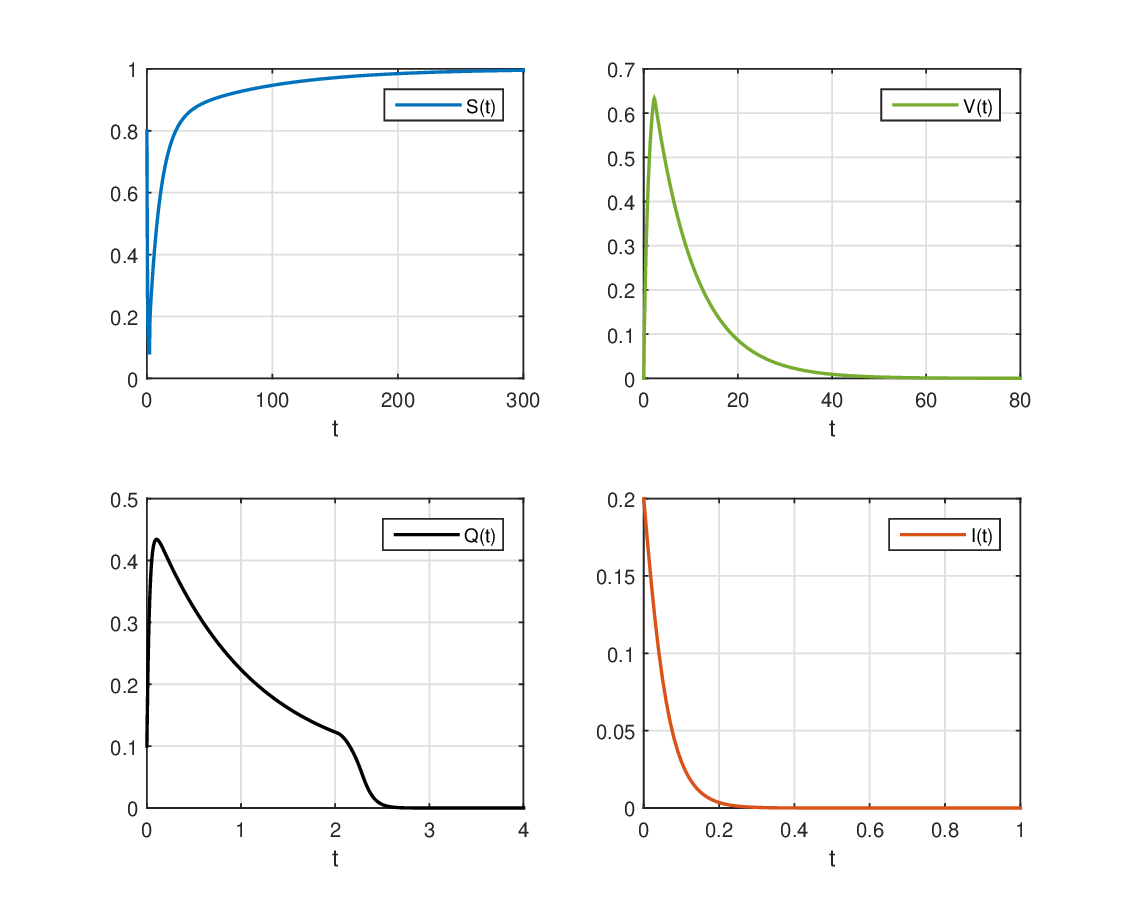}
		\caption{\small Time variation of population with $T = 2$ and $R_0 =0.5<1$. }
		\label{staDFE}
	\end{center}
\end{figure}

In figure $\ref{staEE}$, $T=0.5 \in S$. In this case, each pathological state within the system converges to the steady state as time evolves, implying $ E^* $ is locally asymptotically stable. In figure $\ref{instaEE}$, $T=0.9 \in U$. In this case,  each pathological state within the system converges to a nontrivial periodic solution as time evolves, which implies that $E^*$ is unstable.
\begin{figure}[H]
	\begin{center}
		\includegraphics[width=14cm]{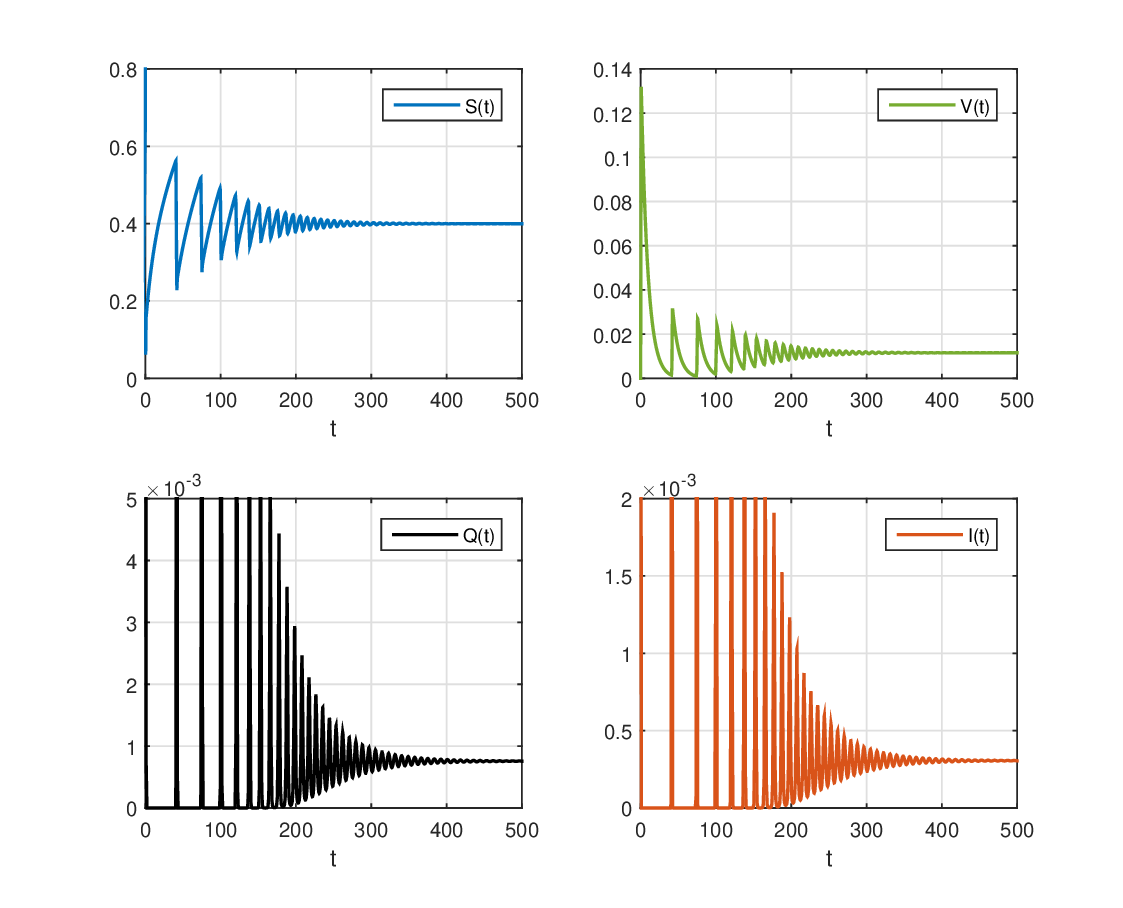}
		\caption{\small Time variation of population with $T = 0.5 \in S$ and $R_0 =2.5>1$. }
		\label{staEE}
	\end{center}
\end{figure}
\begin{figure}[H]
	\begin{center}
		\includegraphics[width=14cm]{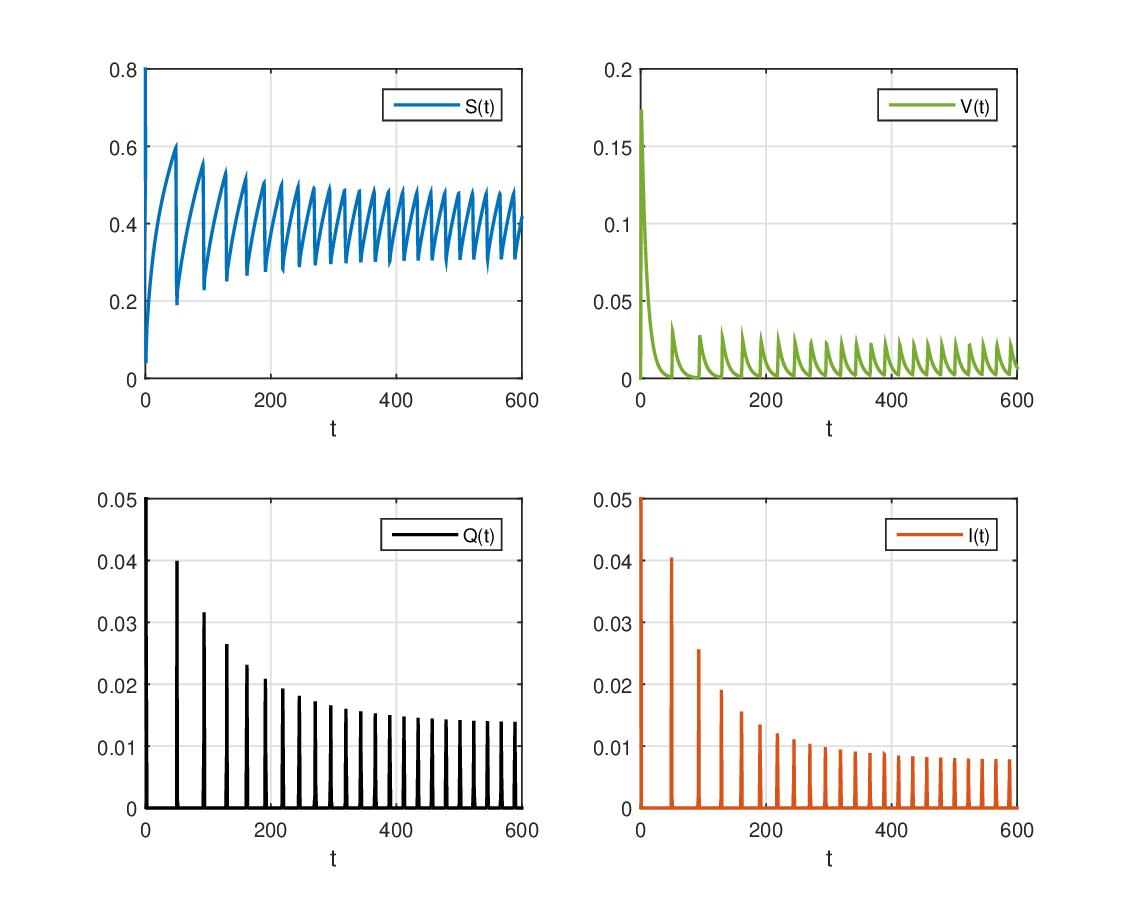}
		\caption{\small Time variation of population with $T = 0.9\in U$ and $R_0 =2.5>1$ }
		\label{instaEE}
	\end{center}
\end{figure}

On figures $\ref{q1}$ and $\ref{q21}$, the stability and instability regions are illustrated by the curves $T_{n}^{+}$ and $T_{n}^{-}$, for $n = 0, 1, 2$ and for different values of $q_1$ or $q_2$ As $q_1$ or $q_2$ increases, the unstable regions expand and the stable regions shrink.   In figure $\ref{q22}$, for another choice of parameters, $R_0 =5$ and $q_1=100$, we draw the same regions. Again, as $q_2$ increases, the unstable regions expand while the stable regions shrink. However, for small $q_2$, the instability regions vanish, implying stability for any $T$. This shows that the endemic equilibrium $E^*$ is stable when $q_2$ is small, and it is destabilized as $q_2$ becomes large. In figure $\ref{d1}$, the stability and instability regions are illustrated by drawing the curves $T_{n}^{+}$ and $T_{n}^{-}$, where $n = 0, 1, 2$, for different values of $\delta$. It is observed that both the stable and unstable regions are barely constant with $\delta$ for high values of this parameter. In figure $\ref{d2}$, these regions are drawn for different values of $\delta$  with $q_2=7$ and $ R_0 =5$. As $\delta$ increases, the stable regions expand while  the unstable regions shrink and vanish. Then, the endemic equilibrium $E^*$ is asymptotically stable, independently of $T$, for large values of $\delta$.
\begin{figure}[H]
	\begin{center}
		\includegraphics[width=11cm]{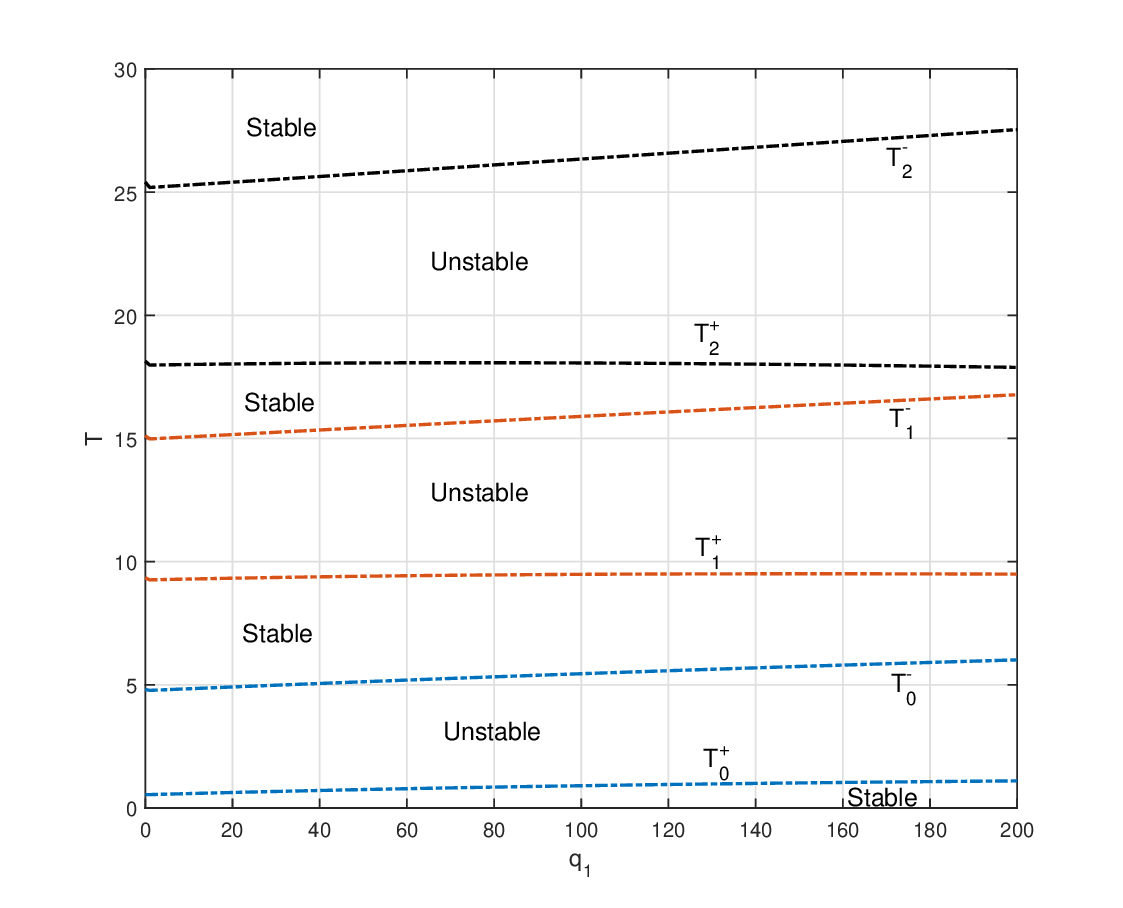}
		\caption{\small The stability and instability regions are illustrated by drawing the curves $T_{n}^{+}$ and $T_{n}^{-}$, where $n = 0, 1, 2$, for different values of $0\leq q_1\leq 200$.}
		\label{q1}
	\end{center}
\end{figure}
\begin{figure}[H]
	\begin{center}
		\includegraphics[width=11cm]{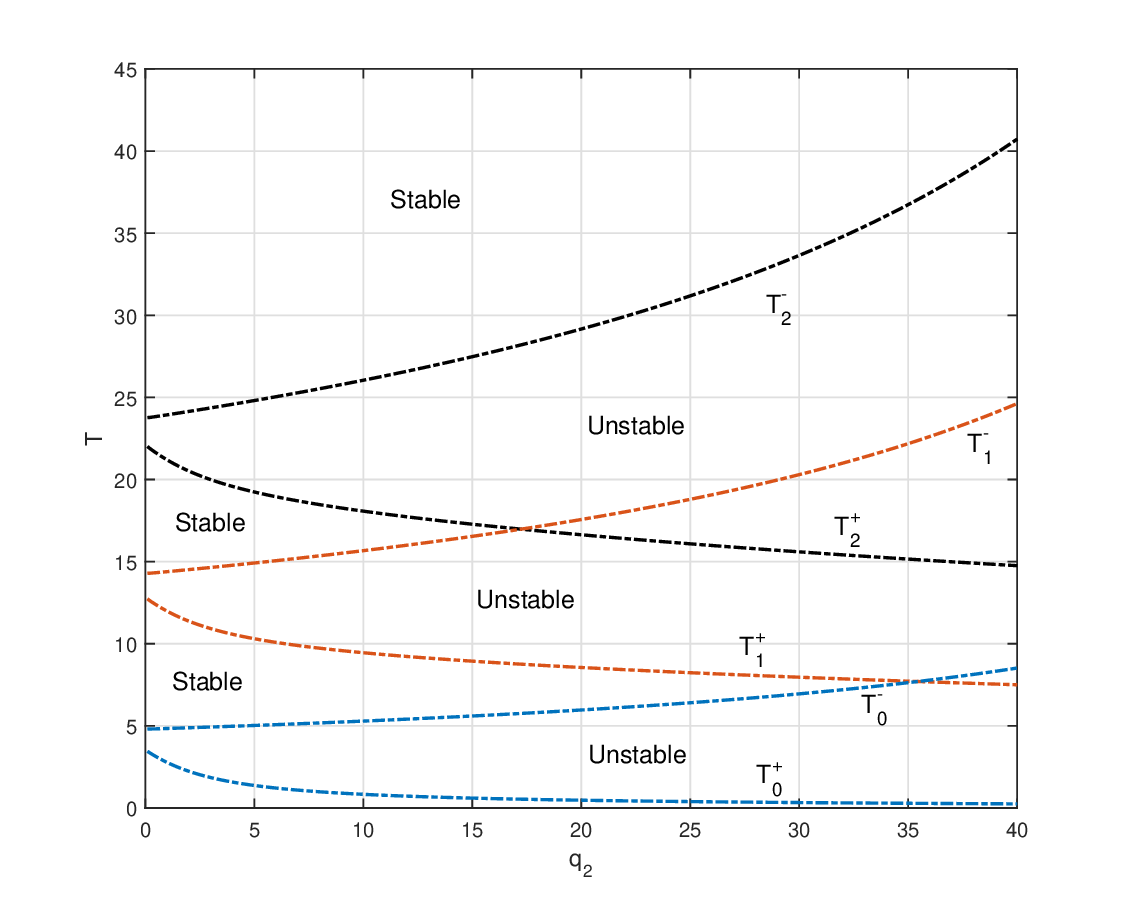}
		\caption{\small  The stability and instability regions are illustrated for different values of $0\leq q_2\leq 40$.}
		\label{q21}
	\end{center}
\end{figure}
\begin{figure}[H]
	\begin{center}
		\includegraphics[width=11cm]{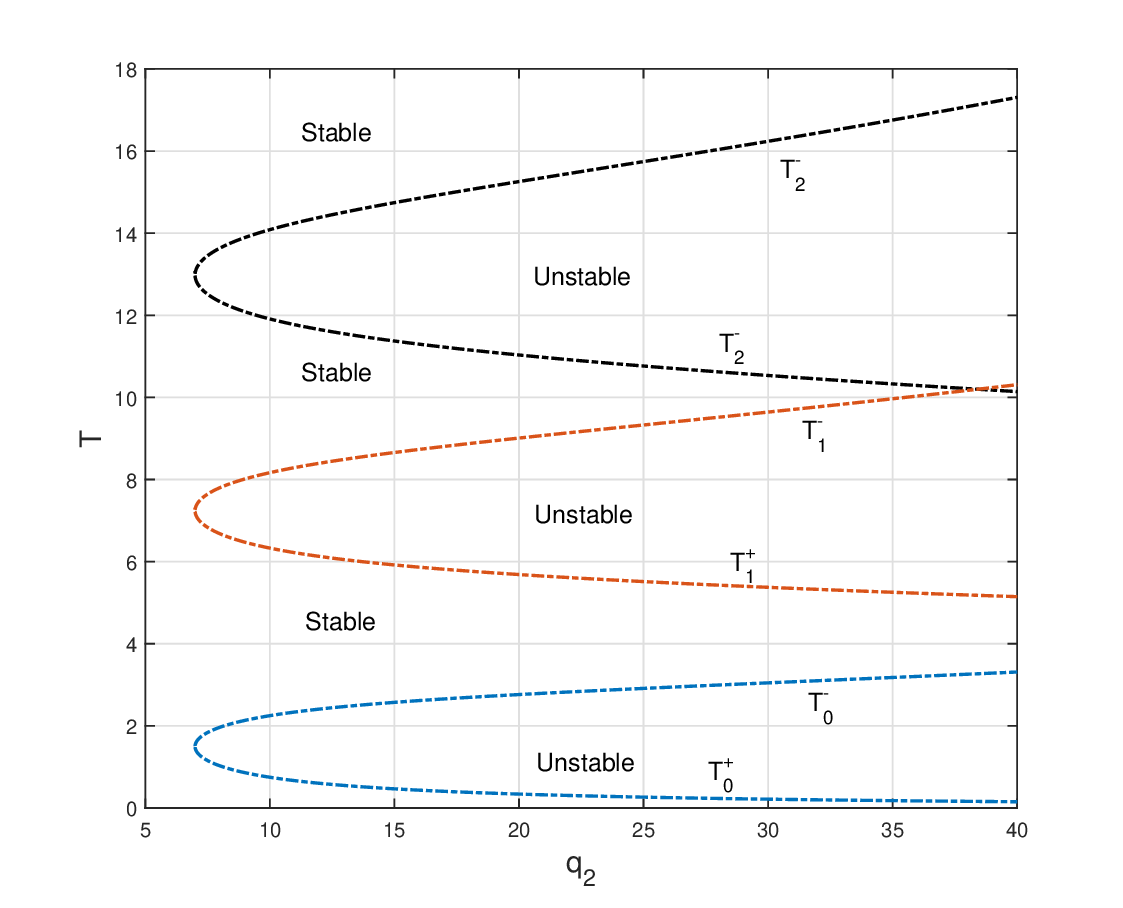}
		\caption{\small The stability and instability regions are illustrated for different values of $0\leq q_2\leq 40$ with $R_0 =5$ and $q_1=100$. }
		\label{q22}
	\end{center}
\end{figure}
\begin{figure}[H]
	\begin{center}
		\includegraphics[width=11cm]{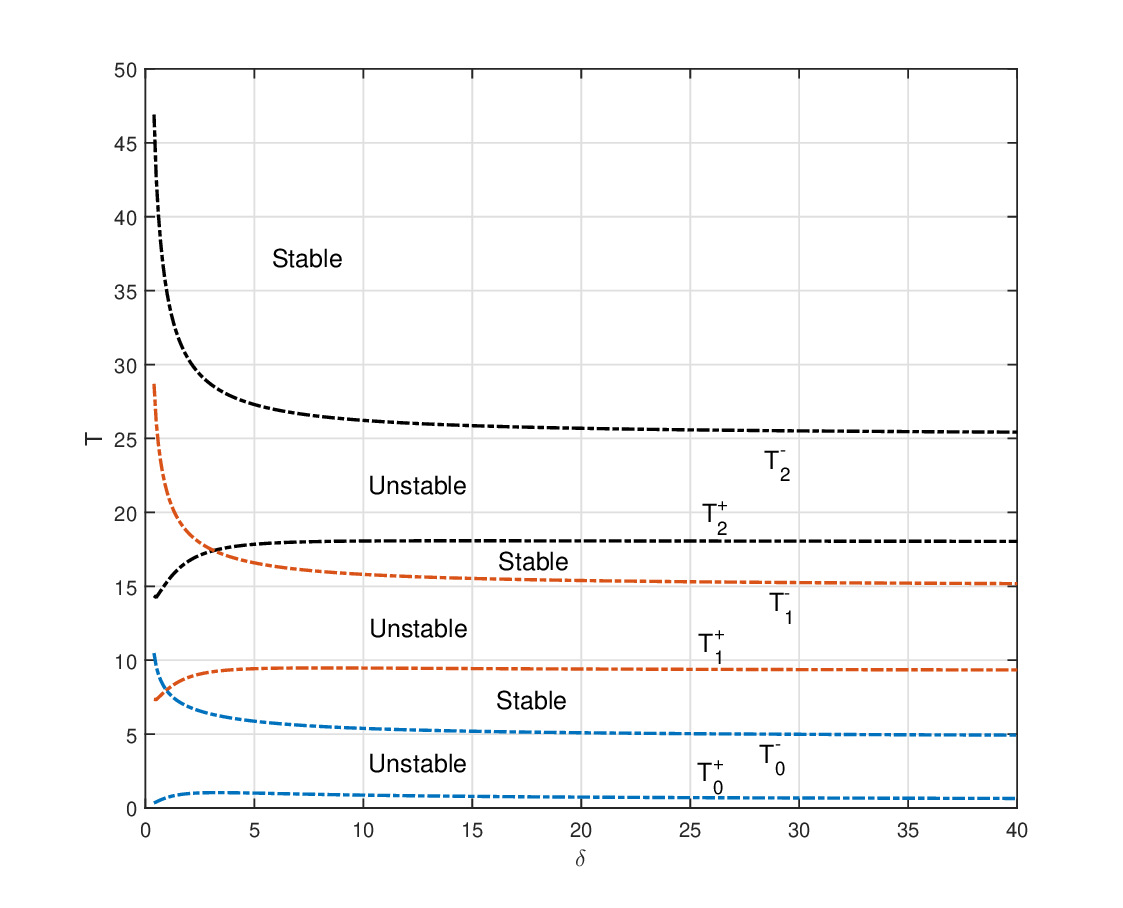}
		\caption{\small The stability and instability regions are illustrated for different values of $0\leq \delta\leq 40$. }
		\label{d1}
	\end{center}
\end{figure}
\begin{figure}[H]
	\begin{center}
		\includegraphics[width=11cm]{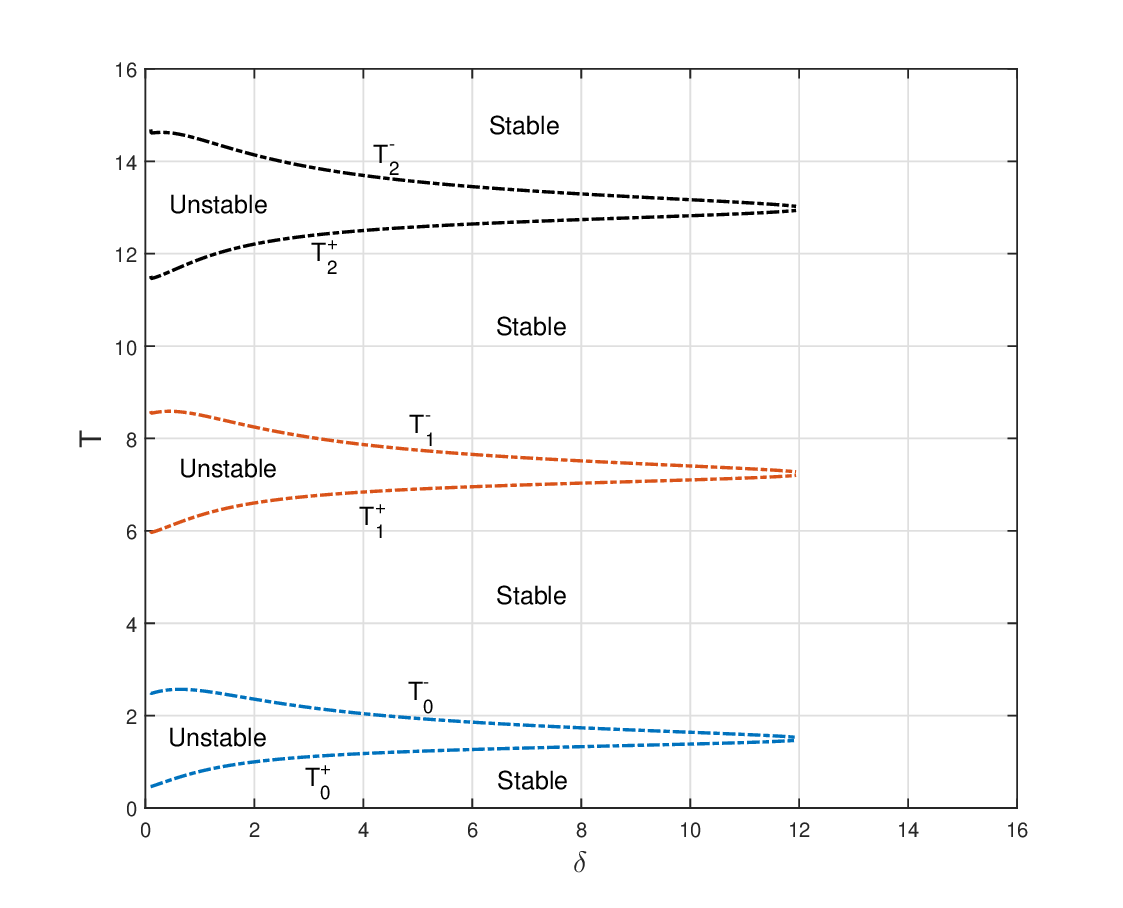}
		\caption{\small The stability and instability regions are illustrated for different values of $0\leq \delta\leq 16$ with $q_2=7$ and $ R_0 =5$. }
		\label{d2}
	\end{center}
\end{figure}\

\begin{figure}[H]
	\begin{center}
		\includegraphics[width=10cm]{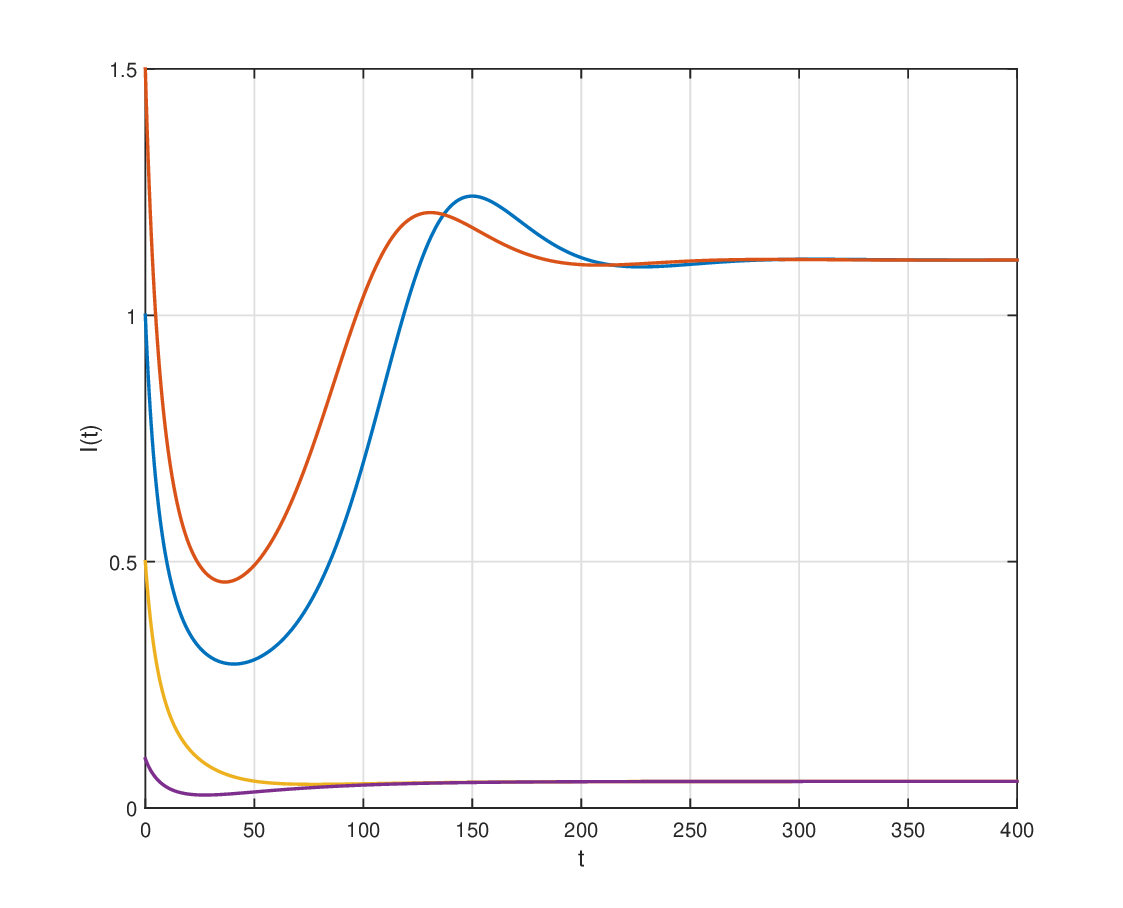}
		\caption{\small An example demonstrating multistability, in the case of multiple endemic equilibria, of the smallest and the largest $\overline{I}$. The middle one seem to be always unstable. Here, we show only the $I$ component. The long-term behavior of the system depends on the initial conditions. The parameters are : $\alpha = 177.46$, $\sigma = 0.5$, $\beta= 0.04$, $\mu= 0.045$, $q_1= 509.99856$, $q_2= 15$, $\lambda= 0.00200000043$, $\delta=140.0001$, $\rho=0.2$,
			$b= 0.52$, $\gamma = 0.1$. The basic reproduction number is $R_0=3.18>1$.}
		\label{d3}
	\end{center}
\end{figure}

In scenarios with multiple endemic equilibria, Figure $\ref{d3}$ illustrates that the trajectory of $I(t)$ can converge to different endemic equilibrium points under appropriate parameter values and initial conditions when $R_0 > 1$. We tested several sets of parameters and the trajectory of $I(t)$ appears to always stabilize at the smallest or the largest equilibrium. Furthermore, despite exploring a wide range of parameter combinations, we did not detect the occurrence of two Hopf bifurcations near each equilibrium. No periodic solutions were observed  when multiple endemic equilibria exist. 

{\color{black}
\section{Discussion}\label{disc}

In this study, we studied a modified SIR mathematical model, based on the one introduced in \cite{kunireccuwave}, incorporating vaccination, quarantine measures for susceptible individuals, and distributed time delay to better describe disease transmission dynamics. Unsurprisingly, the model admits a disease-free equilibrium point $E^0$. Moreover, depending on the parameter values, the model can exhibit one endemic equilibrium, two endemic equilibria, or three endemic equilibria. A quadratic Lyapunov functional is used to demonstrate the global asymptotic stability of the disease-free equilibrium when the basic reproduction number is $R_0\leq 1$. Conversely, if $R_0 > 1$, the disease-free equilibrium becomes unstable. Hence, one of the endemic equilibrium points exhibits stable behavior for certain parameter values, indicating potential disease spread. We emphasize that the analysis of the endemic equilibrium is a challenging task due to the structure of the system and the complexity of the associated characteristic equation. The study is therefore carried out for specific cases. First, we found that when the quarantine sensitivity $q_1$ and the vaccination sensitivity $q_2$ are sufficiently small, or when $q_2$ is small and the parameter $\delta$ is large enough, reflecting less restrictive control measures, there exists a unique endemic equilibrium, which is stable. Additionally, we identified conditions under which a Hopf bifurcation occurs, and simulations revealed periodic solutions, indicating the potential for recurrent epidemic waves under certain combinations of quarantine, vaccination, and time-delay parameters. Mathematically, the transversality condition was verified to determine the direction of the eigenvalue crossing through the imaginary axis. Next, we focused on identifying the regions in which epidemic waves may arise, depending on the key parameters of the model. This allowed us to gain a clear understanding of the shape and structure of these regions. In summary, the analysis shows that the model successfully captures essential features of disease transmission and control, with vaccination and isolation strategies enhancing its capacity to provide meaningful insights into epidemic dynamics, particularly in the presence of distributed time delays. Future work could include an exposed class to better represent latent infection, as well as considering models with diffusion over bounded or unbounded domains. Other possible extensions include a more detailed analysis of the bifurcation structure with distributed delays and the study of more general control or delay mechanisms.
}

\section*{Statements and Declarations}

\noindent There is no conflict of interest.

\noindent There is no data availability

\noindent This research received no external funding.




\end{document}